\documentclass[12pt]{amsart}

\usepackage{amsmath}
\usepackage{mathtools}
\usepackage{amssymb}
\usepackage{amsthm}
\usepackage{amscd}
\usepackage{mathrsfs}
\IfFileExists{dsfont.sty}{\usepackage{dsfont}}{}
\IfFileExists{bbm.sty}{\usepackage{bbm}}{}

\usepackage{graphicx}

\usepackage[dvipsnames]{xcolor}
\usepackage[normalem]{ulem}
\usepackage{comment}

\usepackage{multirow}
\usepackage{bigdelim}
\usepackage{enumerate}
\usepackage{enumitem}
\usepackage[all]{xy}
\usepackage{tikz}

\usepackage[most]{tcolorbox}
\tcbuselibrary{breakable, skins, theorems}
\tcbset{
  mybox/.style={
    colback=white,
    colframe=green!35!black,
    fonttitle=\bfseries,
    breakable=true
  }
}

\PassOptionsToPackage{final}{showkeys}
\usepackage{showkeys} 

\setlist[enumerate]{label=$(\arabic*)$}

\usepackage[colorlinks,linkcolor=red,anchorcolor=blue,citecolor=blue]{hyperref}

\newtheorem{thm}{Theorem}[section]

\newtheorem{cor}[thm]{Corollary}
\newtheorem{prop}[thm]{Proposition}
\newtheorem{conj}[thm]{Conjecture}

\newtheorem{lem}[thm]{Lemma}

\theoremstyle{definition}

\newtheorem{ex}[thm]{Example}

\theoremstyle{remark}
\newtheorem{rem}[thm]{Remark}
\newtheorem{setup}[thm]{Setup}
\newtheorem{obs}[thm]{Observation}

\newtheorem{claim}[thm]{Claim}

\newtheorem{step}{Step}

\numberwithin{equation}{section}

\newcommand{\rk}{\operatorname{rk}}

\newcommand{\Hom}{\mathscr{H}\!\mathit{om}}

\DeclareMathOperator{\Ric}{Ric}

\DeclareMathOperator{\tr}{tr}
\DeclareMathOperator{\Id}{Id}

\DeclareMathOperator{\Homop}{Hom}

\newcommand{\MY}{\operatorname{MY}}
\newcommand{\abs}[1]{\left\lvert#1\right\rvert}
\newcommand{\norm}[1]{\left\|#1\right\|} 
\newcommand{\Rm}{\operatorname{Rm}}
\newcommand{\Cal}{\operatorname{Cal}}
\newcommand{\NA}{\operatorname{NA}}

\newcommand{\cD}{\mathcal{D}}
\newcommand{\cE}{\mathcal{E}}
\newcommand{\cF}{\mathcal{F}}

\newcommand{\cH}{\mathcal{H}}

\newcommand{\cL}{\mathcal{L}}

\newcommand{\cO}{\mathcal{O}}

\newcommand{\cX}{\mathcal{X}}
\newcommand{\cY}{\mathcal{Y}}

\DeclareMathOperator{\pr}{pr}
\DeclareMathOperator{\id}{id}

\newcommand{\reg}{\mathrm{reg}}

\newcommand{\ddbar}{\frac{\sqrt{-1}}{2\pi} \partial \bar{\partial}}

\newcommand{\R}{\mathbb{R}}
\newcommand{\Z}{\mathbb{Z}}
\newcommand{\N}{\mathbb{N}}
\newcommand{\C}{\mathbb{C}}
\newcommand{\Q}{\mathbb{Q}}

\newcommand{\mP}{\mathbb{P}}

\DeclareMathOperator{\E}{\mathbb{E}}

\DeclareMathOperator{\vol}{vol}

\newcommand{\underalign}[2]{\quad \underset{\mathclap{\strut #1}}{#2}\quad}

\title[Miyaoka--Yau inequality and delta invariant I -klt Fano case-]{The Miyaoka--Yau inequality and the delta invariant for Fano varieties I -klt Fano case-}

\author{Tomoyuki HISAMOTO}
\address{Graduate School of Mathematics, Nagoya University, Furocho, Chikusa, Nagoya,
Japan}
\email{{\tt  hisamoto@math.nagoya-u.ac.jp}}

\author{Masataka IWAI}
\address{Department of Mathematics, Graduate School of Science, the University of Osaka,
1-1, Machikaneyama-cho, Toyonaka, Osaka 560-0043, Japan.}
\email{{\tt masataka@math.sci.osaka-u.ac.jp}}
\email{{\tt masataka.math@gmail.com}}

\date{\today}



\begin{document}

\begin{abstract}
We establish the following Miyaoka--Yau inequality for any \(n\)-dimensional klt Fano variety \(X\), possibly K-unstable, in terms of its delta invariant:
\[
\left(2(n+1)\widehat{c}_2(X)-n c_1(X)^2\right)\cdot c_1(X)^{n-2}
\ge -n \left(1-\min\{1,\delta(X)\}\right)^2 \cdot c_1(X)^n.
\]
Moreover, if equality holds in this inequality, then there exist a finite quasi-\'etale cover $\widehat{X} \to X$ and a positive integer $p$ such that $\widehat{X} \cong \mathbb{P}(1,\ldots, 1, p)$. We also discuss the relationship between Donaldson's conjecture and the Miyaoka--Yau inequality.
\end{abstract}

\maketitle

\setcounter{tocdepth}{2}

\makeatletter
\renewcommand{\tocsection}[3]{%
  \indentlabel{\@ifnotempty{#2}{\ignorespaces#1 #2.\quad}}#3%
}
\renewcommand{\tocsubsection}[3]{%
  \hspace{0.75em}\indentlabel{\@ifnotempty{#2}{\ignorespaces#1 #2.\quad}}#3%
}
\renewcommand{\tocsubsubsection}[3]{%
  \hspace{1.5em}\indentlabel{\@ifnotempty{#2}{\ignorespaces#1 #2.\quad}}#3%
}
\makeatother

\tableofcontents

\medskip
\section{Introduction}
\label{sec-Introduction}
Using his solution of the Calabi conjecture in \cite{Yau78}, Yau \cite{Yau77} proved that every \(n\)-dimensional compact K\"ahler manifold with ample canonical bundle satisfies
\begin{equation}
\label{equa-intro-yau}
    \big(2(n+1)c_2(X)-n c_1(X)^2\big) \cdot c_1(K_X)^{n-2} \geq 0.
\end{equation}
Independently, Miyaoka \cite{Miy77} proved that $3c_2(X)-c_1(X)^2 \geq 0$ for every smooth projective surface $X$ with big canonical bundle. 
The inequality \eqref{equa-intro-yau} is now called the \emph{Miyaoka--Yau inequality}. 
When the canonical divisor is positive, extensions to klt varieties and log pairs have been studied in \cite{GKPT19b, GT22, CGG23, ZZZ25, IJZ25}.
 In particular, \cite{IJZ25} established the Miyaoka--Yau inequality for klt varieties of arbitrary dimension with big canonical divisor by using non-pluripolar products.

In contrast, the Miyaoka--Yau inequality does not hold for arbitrary Fano varieties. For example, the toric Fano $4$-fold $\mathbb{P}_{\C\mathbb{P}^3}(\mathcal{O}_{\C\mathbb{P}^3} \oplus \mathcal{O}_{\C\mathbb{P}^3}(3))$ does not satisfy it (see also Example \ref{ex-counter-example-MY}). 
One difficulty is that a Fano manifold need not admit a K\"ahler--Einstein metric.

If a Fano manifold admits a K\"ahler--Einstein metric, then the Miyaoka--Yau inequality holds by \cite{CO75}. 
More generally, it was proved for K-semistable Fano varieties in \cite{Bando, SW16, GKP22, DGP24}.
By \cite{FO18, BJ20}, K-semistability is characterized by the delta invariant being at least \(1\). 
For a Fano manifold, the greatest Ricci lower bound is
$\min\{1,\delta(X)\}$ by \cite{CRZ19, BBJ21}.
Therefore, it is natural to expect a relation of the Miyaoka--Yau inequality to the delta invariant, or 
to the greatest Ricci lower bound. The first main result of this paper confirms this expectation.

\begin{thm}[{=Theorem \ref{thm-MY-delta}}]
\label{thm-main-klt}
Let $X$ be an $n$-dimensional projective klt Fano variety. Then the  following Miyaoka--Yau inequality holds:
\begin{equation}
\label{eq-main-MY-delta}
\left( 2(n+1)\widehat{c}_2(X) - n c_1(X)^2\right) \cdot  c_1(X)^{n-2}
\ge-n(1 - \delta'(X))^2 
\cdot c_1(X)^{n}
\end{equation}
 where $\delta(X)$ is the delta invariant and $\delta'(X) \coloneqq \min\{1, \delta(X)\}$. 
 
Moreover, if equality holds in \eqref{eq-main-MY-delta}, then there exist a finite quasi-\'etale cover $\widehat{X} \to X$ and a positive integer $p \in \N$ such that $\widehat{X} \cong \mathbb{P}(1,\ldots, 1, p)$.

In particular, if $X$ is smooth and equality holds in \eqref{eq-main-MY-delta}, then $X \cong \C\mathbb{P}^n$.
\end{thm}
Here, $\widehat{c}_2(X)$ denotes the orbifold Chern class introduced in \cite{Kaw92, GKPT19b, GK20}. When $X$ is smooth, this class coincides with the usual Chern class. If $X$ is K-semistable, then $\delta'(X)=1$, and hence the results of \cite{Bando, SW16, GKP22, DGP24} follow from Theorem~\ref{thm-main-klt}. 
The converse of Theorem~\ref{thm-main-klt} also holds. Indeed, $\mathbb{P}(1,\ldots, 1, p)$ satisfies equality in \eqref{eq-main-MY-delta} by Example~\ref{ex-weighted-projective}, and this equality is preserved under finite quasi-\'etale covers by \cite{GKPT19b, Zhu21, LZ22}.
We remark that, assuming the existence of a smooth divisor $D \in \lvert -K_X \rvert$, \cite{SW16} employed conical K\"ahler--Einstein metrics to derive a   
Chern number inequality which is similar to Theorem \ref{thm-main-klt}. 
We also extend the inequality to projective klt weak Fano
varieties, that is, when $-K_X$ is nef and big, and describe
the equality case in terms of the anticanonical morphism;
see Theorem \ref{thm:weak-fano-miyaoka-yau}.

We now explain the idea behind the proof of Theorem \ref{thm-main-klt}. It is based on Tian's approach \cite{Tian92}, which interprets the Miyaoka--Yau inequality as a Bogomolov--Gieseker inequality for the canonical extension sheaf.
Replacing the Bogomolov--Gieseker inequality with Langer's inequality \cite{Lan04} yields the corresponding Miyaoka--Yau inequalities involving the delta invariant. 

From this perspective, Theorem \ref{thm-main-klt} in the smooth case follows by combining Tian's slope estimate for the canonical extension sheaf \cite{Tian92} with Langer's inequality \cite{Lan04}.
In the klt case, Langer's inequality was established in \cite{IMM24}.
 During the preparation of this paper, Tian's slope estimate was extended to klt varieties, and more generally to log Fano pairs, in \cite[Lemma 5.9]{XZ26}. The proof reduces the general case to the K-semistable case established in \cite[Appendix B]{XZ26}.
We therefore apply \cite{XZ26} in the proof of Theorem \ref{thm-main-klt}. However, we also include our original argument, based on the method of \cite{DGP24}, in Appendix \ref{sec-proof-DGP}. 
The structure in the equality case is established by using \cite{IMM24} to analyze the canonical extension sheaf.

In Section \ref{sec-Calabi}, we give a differential-geometric
proof of the following conditional result, relating the
Calabi functional to the delta invariant.

\begin{thm}[{=Theorem \ref{thm-Donaldson-MY}}]
\label{thm-main-Donaldson-MY}
Let $X$ be an $n$-dimensional Fano manifold with $n\ge2$.
Assume that Donaldson's conjecture
(Conjecture \ref{Donaldson's conjecture}) holds for $(X,-K_X)$. 
Then the inequality \eqref{eq-main-MY-delta} holds.
\end{thm}

Although this result is weaker than Theorem \ref{thm-main-klt},
we state it as a theorem because its proof provides a new
differential-geometric perspective on the Miyaoka--Yau inequality.

\medskip
\noindent{\bfseries Acknowledgments.}
The authors thank the organizers of the ``31st Symposium on Complex Geometry'' held in Japan, where discussions provided the starting point for this work. 
They also thank Prof. S\'ebastien Boucksom for kindly answering their questions.
 T.\,H.\ thanks him and Prof. Siarhei Finski for helpful discussions concerning Section \ref{sec-Calabi}. 
M.\,I.\ also thanks Prof. Joaqu\'in Moraga for raising the question on weak Fano varieties discussed in Subsection \ref{subsec-weak-fano}.
M.\,I.\ was supported by the Grant-in-Aid for Early Career Scientists, No.\ 22K13907. 
T.\,H.\ is supported by JSPS Grant-in-Aid for Scientific Research $(\mathrm{C})$, No.\ 26K06811. 

\medskip
\noindent{\bfseries Use of AI.}
{
The authors used ChatGPT 5.5-5.6 Sol, ChatGPT 6 Astra, and Gemini 3.1
to improve the writing, search the literature, and discuss
and check mathematical arguments and calculations.
The authors reviewed all AI-assisted material included in
the paper and checked the arguments, calculations, and
references themselves.
The authors take full responsibility for the results and
proofs in this paper and for citing previous work accurately
and fully.
}

\medskip
\noindent{\bfseries Notation and conventions.}
In this paper, we denote by $\mathbb{N}$ the set of non-negative integers. 
Unless otherwise stated, we work over the complex number field $\C$. 
For any real vector space $M$, we write $M^\vee$ for its dual.
We employ the standard notation and conventions in \cite{Har77, KM98}.
 Additionally, we say that a normal analytic variety $X$ is a \emph{klt variety} if the pair $(X, 0)$ is klt; equivalently, $X$ is log terminal.

Unless stated otherwise, all sheaves considered in this paper are assumed to be coherent.
For any torsion-free (coherent) sheaf $\mathcal{E}$, we define the dual reflexive sheaf $\mathcal{E}^\vee$ by
\(
\mathcal{E}^\vee \coloneqq \Hom(\mathcal{E}, \mathcal{O}_X).
\)
We also identify ``locally free sheaves of finite rank'' with ``vector bundles.''
When working on smooth complex varieties, we use the terms ``(Cartier) divisors'' and ``line bundles'' interchangeably.

\section{Preliminaries}
\label{sec-preliminary}

\subsection{Canonical extension sheaf}
We first recall the notion of the \emph{canonical extension sheaf} (cf.~\cite[Section 2]{Tian92}, \cite[Section 4]{GKP22}, \cite[Section 3]{DGP24}). In order to use it also in Section \ref{sec-proof-of-klt}, we adopt the definition for singular varieties as in \cite[Section 4]{GKP22}.
Let $X$ be a normal variety, and assume that $K_X$ is $\Q$-Cartier. Then, by \cite[Section 4]{GKP22}, we have
$c_1(-K_X) \in H^{1}(X, \Omega_{X}^{1})$.
Since $H^1(X, \Omega_{X}^{1}) = \operatorname{Ext}^1(\mathcal{O}_{X}, \Omega_{X}^{1})$, this gives rise to an extension of $\mathcal{O}_{X}$ by $\Omega_{X}^{1}$:
$$
0 \longrightarrow \Omega_{X}^{1} \longrightarrow \mathscr{E} \longrightarrow \mathcal{O}_{X} \longrightarrow 0.
$$
Taking the dual of this sequence, we obtain
\begin{equation}
\label{eq-defn-canonical-extension}
0 \longrightarrow \mathcal{O}_{X} \longrightarrow \mathscr{V}\longrightarrow \mathcal{T}_{X} \longrightarrow 0.
\end{equation}
Here $\mathcal{T}_{X}$ denotes the reflexive tangent sheaf.

In the above construction, for any positive real number $\lambda >0$, we may consider the extension corresponding to $\lambda c_1(-K_X) \in H^{1}(X, \Omega_{X}^{1})= \operatorname{Ext}^1(\mathcal{O}_{X}, \Omega_{X}^{1})$. 
As in \eqref{eq-defn-canonical-extension}, one can then construct a reflexive sheaf $\mathscr{V}_{\lambda}$.
We call the sheaf $\mathscr{V}_{\lambda}$ the \emph{canonical extension sheaf with extension class $\lambda c_1(X)$.} If the coefficient $\lambda$ is not specified, we simply call it the \emph{canonical extension sheaf}.

\subsection{Orbifold second Chern classes }
We briefly recall the orbifold Chern classes (or $\Q$-Chern classes) used in \cite{K++, GKPT19b}. For details, see \cite[Section 2]{IMM24}. We follow the notation of \cite[Section 2]{IMM24}.

Let $X$ be a projective klt variety of dimension $n \ge 2$.
For any reflexive sheaf $\mathcal{E}$,
there is a symmetric $\Q$-multilinear form satisfying properties (P1) and (P2) in \cite[Subsection 2.3]{IMM24}:
$$
\widehat{c}_2(\mathcal E)\colon
\mathrm{N}^1(X)_{\mathbb Q}^{n-2}
\longrightarrow \mathbb Q
\qquad
(\alpha_1,\ldots,\alpha_{n-2})
\longmapsto
\widehat{c}_2(\mathcal E)\alpha_1\cdots\alpha_{n-2}.
$$
We call $\widehat{c}_2(\mathcal E)$ the orbifold second Chern class.
Similarly, the orbifold first Chern class $c_1(\mathcal{E})$ and its square $c_1(\mathcal{E})^2$ are defined as symmetric $\Q$-multilinear forms. 
In particular, for the reflexive tangent sheaf $\mathcal{T}_X \coloneqq  (\Omega_{X}^{[1]})^{\vee}$ of $X$,
we write
$$
c_1(X)\coloneqq c_1(\mathcal{T}_X )
\quad \text{and} \quad
\widehat{c}_2(X)\coloneqq \widehat{c}_2(\mathcal{T}_X )
$$

Following \cite[Definition 2.2]{IMM24},
for a Weil divisor $D$ on $X$ and its associated rank $1$ reflexive sheaf $\mathcal F\coloneqq \mathcal O_X(D)$, we define $c_1(D)\coloneqq \widehat{c}_1(\mathcal F)$. If $D$ is $\Q$-Cartier, choose $m\in\mathbb Z_+$ such that $mD$ is Cartier. Then the
orbifold first Chern class $c_1(D)$ coincides with the $\Q$-multilinear form $\frac{1}{m}c_1(mD)$.
For a torsion-free sheaf $\mathcal E$ on $X$, we also define
$c_1(\mathcal E)\coloneqq c_1(\mathcal E^{\vee\vee})$.
By \cite[Definition 2.2]{IMM24}, we have $c_1(\mathcal E)=c_1(\det\mathcal E)$.
In particular, since $K_X$ is $\Q$-Cartier, we have $c_1(X) = c_1(-K_X)$ in the above notation.

To simplify notation,
for torsion-free sheaves $\mathcal E, \mathcal{F}$ and a rational number $\lambda \in \Q$,
we write $c_1(\mathcal{E})= \lambda  c_1(\mathcal{F})$ to mean that
$$
\left(  c_1(\mathcal{E}) -\lambda c_1(\mathcal{F}) \right) \cdot \alpha_1 \cdots \alpha_{n-1}
=0
\quad \forall \alpha_1,\ldots,\alpha_{n-1} \in \mathrm{N}^1(X)_{\mathbb Q}
$$
holds.
If there are $\Q$-Cartier Weil divisors $E, F$ such that $\mathcal{E} = \mathcal{O}_X(E), \mathcal{F} = \mathcal{O}_X(F)$, then $ c_1(\mathcal{E})= \lambda  c_1(\mathcal{F})$ is equivalent to the numerical equivalence $E \equiv \lambda F$.

Finally, a reflexive sheaf $\mathcal{E}$ of rank $1$ is called a $\Q$-line bundle if
its reflexive tensor power $(\mathcal{E}^{\otimes m})^{\vee\vee}$ is locally free for some positive integer $m$.
For a Weil divisor $E$ with $\mathcal{E} = \mathcal{O}_X(E)$, this is equivalent to $E$ being $\Q$-Cartier.

\section{Miyaoka--Yau inequality for klt Fano varieties}
\label{sec-proof-of-klt}

\subsection{Proof of Theorem \ref{thm-main-klt} }

The following theorem was proved in the smooth case in \cite{Tian92} and in the singular case in \cite{XZ26}. In Appendix \ref{sec-proof-DGP}, we briefly explain that it can also be obtained by the same method as in \cite{DGP24}.

\begin{thm}[{\cite[Section 4]{Tian92}, \cite[Lemma 5.9]{XZ26}}]
\label{thm-slope-gamma-klt}
Let $X$ be an $n$-dimensional klt Fano variety and  let $\mathscr{V}$ be the canonical extension sheaf. 
For any  torsion-free sheaf $\mathcal{F} \subset \mathscr{V}$ of rank $r$, 
$$
 c_1(\mathcal{F}) \cdot c_1(X)^{n-1} 
\le
\left( 1- \delta'(X) \left(1-\frac{r}{n+1}\right) \right)c_1(X)^{n}.
$$
Here $\delta(X)$ denotes the delta invariant, and we set $\delta'(X) \coloneqq \min\{1, \delta(X)\}$.
In particular, for any torsion-free quotient sheaf $\mathscr{V} \twoheadrightarrow \mathcal Q$ of rank $q$, 
\[
 c_1(\mathcal Q) \cdot c_1(X)^{n-1} 
 \ge\frac{ \delta'(X)q}{n+1}c_1(X)^{n}.
\]
\end{thm}
Indeed, $\mu^{\min}_{-K_X}(\mathscr{V})=-\mu^{\max}_{-K_X}(\mathscr{V}^{\vee})\ge\frac{\delta'(X)}{n+1}c_1(X)^{n}$ holds by \cite[Lemma 5.9]{XZ26}, and hence Theorem \ref{thm-slope-gamma-klt} follows immediately from a straightforward computation. 
Here, $\mu_{-K_X}(\mathscr{V})$ denotes the usual slope of $\mathscr{V}$ with respect to $(-K_X)^{n-1}$, while $\mu^{\min}_{-K_X}(\mathscr{V})$ (resp.\ $\mu^{\max}_{-K_X}(\mathscr{V})$) denotes its minimal slope (resp.\ maximal slope) in the usual sense.

\begin{thm}[{=Theorem \ref{thm-main-klt}}]
\label{thm-MY-delta}
Let $X$ be an $n$-dimensional projective klt Fano variety.
Then the following Miyaoka--Yau inequality holds:
\begin{equation}
\label{eq-MY-delta-inequality}
\left( 2(n+1)\widehat{c}_2(X) - n c_1(X)^2\right) \cdot c_1(X)^{n-2}
\ge -n(1 - \delta'(X))^2 \cdot c_1(X)^{n}.
\end{equation}
Moreover, the following statements hold in the case of equality.
\begin{enumerate}
\item If equality holds in \eqref{eq-MY-delta-inequality}, then there exist a finite quasi-\'etale cover $\widehat{X} \to X$ and a positive integer $p \in \N$ such that $\widehat{X} \cong \mathbb{P}(1,\ldots, 1, p)$.
\item If $X$ is smooth and equality holds in \eqref{eq-MY-delta-inequality}, then $X \cong \C\mathbb{P}^n$.
\end{enumerate}

\end{thm}

\begin{proof}
Throughout the proof, we assume $n \ge 2$. We divide the proof into several steps.

\begin{step}[Proof of inequality \eqref{eq-MY-delta-inequality}]

Let $\mathscr{V}$ be the canonical extension sheaf in Theorem \ref{thm-slope-gamma-klt}, and let $\mathcal{F} \subset \mathscr{V}$ be the $(-K_X)^{n-1}$-maximally destabilizing sheaf of rank $r$.
By Theorem \ref{thm-slope-gamma-klt}, we obtain
\begin{equation}
\label{eq-slope-estimate-caonical-ext}
\begin{aligned}
&\mu^{\max}_{-K_X}(\mathscr{V}) \le
\frac{1}{r}\left( 1-\delta'(X)\left(1-\frac{r}{n+1}\right) \right)c_1(X)^{n},
\\
&\mu^{\min}_{-K_X}(\mathscr{V}) \ge \delta'(X) \frac{c_1(X)^n}{n+1}
\quad \text{and} \quad
\mu_{-K_X}(\mathscr{V}) = \frac{c_1(X)^n}{n+1}.
\end{aligned}
\end{equation}
We recall Langer's inequality, which is \cite[Theorem 5.1]{Lan04} in the smooth case and
\cite[Proposition 3.4]{IMM24}  in the klt case.
Since the canonical extension is locally split, we have $c_1(\mathscr V)=c_1(X)$ and $\widehat c_2(\mathscr V)=\widehat c_2(X)$.
Applying \cite[Proposition 3.4]{IMM24} with $D=H=-K_X$, $\mathcal E=\mathscr{V}$, and $\rk \mathscr{V}=n+1$, we obtain
\begin{equation}
\label{eq-Miyaoka-Yau-delta-ineq}
\begin{aligned}
&\left( 2(n+1)\widehat{c}_2(X) - n c_1(X)^2\right) \cdot c_1(X)^{n-2} \\
&\underalign{\raisebox{-0.5ex}{\tiny \text{\cite[Prop. 3.4]{IMM24}}}}{\ge}
- \frac{(n+1)^2}{c_1(X)^{n}}
  \bigl(\mu^{\max}_{-K_X}(\mathscr{V}) - \mu_{-K_X}(\mathscr{V}) \bigr)
  \bigl( \mu_{-K_X}(\mathscr{V}) - \mu^{\min}_{-K_X}(\mathscr{V})\bigr)\\
& \underalign{\text{\eqref{eq-slope-estimate-caonical-ext}}}{\ge}
  - \frac{n+1-r}{r}(1 -\delta'(X))^2 \cdot c_1(X)^{n}
  \underalign{\text{(by $r\ge1$)}}{\ge}
  -n (1 - \delta'(X) )^2 \cdot c_1(X)^{n}.
\end{aligned}
\end{equation}
This proves the desired inequality.

\end{step}
\begin{step}[The equality case of \eqref{eq-MY-delta-inequality} in the semistable case]
\label{step-MY-equality-2}

We next consider the case where equality holds in \eqref{eq-MY-delta-inequality}.
By \cite[Theorem 1.14]{GKP16}, after passing to a finite quasi-\'etale cover, we may assume that $X$ is maximally quasi-\'etale, i.e., that the natural homomorphism \(\widehat{\pi}_{1}(X_{\reg}) \to \widehat{\pi}_{1}(X)\) is an isomorphism.
Here, $X_{\reg}$ denotes the smooth locus of $X$, and $\widehat{\pi}_{1}(\bullet)$ denotes the \'etale fundamental group.
Note that $\delta'(X)$ is invariant under finite quasi-\'etale  covers by \cite[Corollary 4.13]{Zhu21} or  \cite[Theorem 1.2]{LZ22}.

Suppose that $\mathscr{V}$ is $(-K_X)^{n-1}$-slope semistable.
Equality in \eqref{eq-MY-delta-inequality} gives
\[
0 \underset{\text{(BG ineq.)}}{\le}
\left( 2(n+1)\widehat{c}_2(X) - n c_1(X)^2\right) \cdot c_1(X)^{n-2}
= -n(1 - \delta'(X))^2 \cdot c_1(X)^{n}
\le 0
\]
It follows that $\delta'(X)=1$, and hence $X$ is $K$-semistable.
Therefore, \cite[Theorem 1.3]{GKP22} and \cite[Theorem 9]{DGP24} imply that $X \cong \C\mathbb{P}^n$.

\end{step}

\begin{step}[The structure of the filtration of $\mathscr{V}$]
\label{step-MY-equality-3}
From now on, assume that $X$ is maximally quasi-\'etale and that
$\mathscr{V}$ is not $(-K_X)^{n-1}$-slope semistable.
Thus we also assume that $\delta(X) = \delta'(X) \le 1$. In particular, $\delta'(X)$ is rational by \cite[Theorem 1.2]{LXZ22}.

\begin{claim}
\label{claim-str-HN-filtration}
The $(-K_X)^{n-1}$-Harder--Narasimhan filtration of the canonical extension sheaf $\mathscr{V}$ is given by
\begin{equation}
\label{eq-defn-filtration-V}
0 \to \mathcal{F} \to \mathscr{V}\to \mathcal{Q} \to 0
\end{equation}
and satisfies the following properties:
\begin{itemize}
\item[$(\alpha)$] $\mathcal{F}$ is a $\Q$-line bundle with $c_1(\mathcal{F})=\frac{n+1-n \delta'(X)}{n+1}c_1(X)$.
\item[$(\beta)$] $\mathcal{Q}$ is a reflexive sheaf of rank $n$, and the above exact sequence \eqref{eq-defn-filtration-V} is analytically locally split.
\item[$(\gamma)$]
There exists a $\Q$-line bundle $\mathcal{M}$ such that $\mathcal{Q} \cong \mathcal{M}^{\oplus n}$ and $c_1(\mathcal{M})=\frac{\delta'(X)}{n+1}c_1(X)$.
\end{itemize}
\end{claim}

\begin{proof}[Proof of Claim \ref{claim-str-HN-filtration}]
By \cite[Proposition 3.4]{IMM24}, the Harder--Narasimhan filtration of $\mathscr{V}$ with respect to $(-K_X)^{n-1}$ has length $2$.

We first prove assertion $(\alpha)$.
Since equality holds in both \eqref{eq-slope-estimate-caonical-ext} and \eqref{eq-Miyaoka-Yau-delta-ineq}, we have $r=\rk \mathcal{F}=1$ and
\begin{equation}
\label{eq-Miyaoka-Yau-eq-1}
\mu_{-K_X}(\mathcal{F})=
\mu^{\max}_{-K_X}(\mathscr{V})
=\frac{n+1 - n\delta'(X)}{n+1}c_1(X)^n,
\quad
\mu_{-K_X}(\mathcal{Q})=
\mu^{\min}_{-K_X}(\mathscr{V})
=\frac{\delta'(X)}{n+1}c_1(X)^n
\end{equation}
At this point, we only know that $\mathcal{F}$ is a reflexive sheaf of rank $1$ and that $\mathcal{Q}$ is torsion-free.
By \cite[Proposition 3.4 (1), (2)]{IMM24}, there exists $\lambda \in \Q$ such that
\begin{equation}
\label{eq-Miyaoka-Yau-eq-2}
\left( 2n \widehat{c}_2(\mathcal{Q}^{\vee\vee}) - (n-1)c_1(\mathcal{Q}^{\vee\vee})^2 \right)\cdot (-K_{X})^{n-2}
=
0,
\quad
c_1(\mathcal{F}) - \frac{1}{n}c_1(\mathcal{Q})
=
\lambda c_1(-K_X)
\end{equation}
In particular, \eqref{eq-Miyaoka-Yau-eq-1} and \eqref{eq-Miyaoka-Yau-eq-2} yield
$\lambda=1-\delta'(X)$.
Combining $c_1(\mathcal{F})+c_1(\mathcal{Q})=c_1(-K_X)$ with \eqref{eq-Miyaoka-Yau-eq-2}, we obtain
$c_1(\mathcal{F}) = \frac{n+1-n \delta'(X)}{n+1}c_1(X)$.
Since equality also holds in \cite[Proposition 3.4 (3.5)]{IMM24}, a direct calculation gives $$
\left(c_1(\mathcal{F}) - \frac{n+1-n \delta'(X)}{n+1}c_1(X) \right)^2
(-K_X)^{n-2}
\underalign{\text{(cf. \cite[(3.5)]{IMM24})}}{=}0
$$
Thus it follows from \cite[Lemma 2.5]{IMM24} that $\mathcal{F}$ is a $\Q$-line bundle with $c_1(\mathcal{F})  = \frac{n+1-n \delta'(X)}{n+1}c_1(X)$.

We next prove assertion $(\beta)$.
Since $\mathcal{F}$ is a $\Q$-line bundle, $\det \mathcal{Q}$ is also a $\Q$-line bundle. Moreover, both $\mathcal F$ and $\mathcal{Q}^{\vee\vee}$ are numerically projectively flat in the sense of \cite[Definition 2.9]{IMM24}, since $\mathcal{Q}^{\vee\vee}$ is $(-K_X)^{n-1}$-slope semistable and satisfies \eqref{eq-Miyaoka-Yau-eq-2}.
Moreover, a direct calculation shows that
\[
\bigl(\widehat{c}_2(\mathscr{V})-\widehat{c}_2(\mathcal{F})-\widehat{c}_2(\mathcal{Q}^{\vee\vee})-c_1(\mathcal{F})c_1(\mathcal{Q})\bigr)\cdot (-K_X)^{n-2}=0
\]
Thus \cite[Lemma 2.14]{IMM24} implies that the above exact sequence \eqref{eq-defn-filtration-V} is analytically locally split and that $\mathcal{Q}$ is reflexive.

We now prove assertion $(\gamma)$.
Since $\mathcal{Q}$ is numerically projectively flat, \cite[Proposition 1.6]{GKP22} shows that $\mathcal{Q}|_{X_{\reg}}$ is locally free and projectively flat.
Since a klt Fano variety is simply connected (cf. \cite[Theorem 1.1]{Tak00}),
the group
\(
\widehat{\pi_1}(X_\reg) \cong \widehat{\pi_1}(X)
\)
is trivial.
Hence, by \cite[Proposition 3.10]{GKP22}, there exists an invertible sheaf $\mathcal{M}'$ on $X_{\reg}$ such that $\mathcal{Q}|_{X_{\reg}} \cong {\mathcal{M}'}^{\oplus n}$.
Setting $\mathcal{M}\coloneqq (i_{*}\mathcal{M}')^{\vee\vee}$, we obtain $\mathcal Q \cong \mathcal M^{\oplus n}$ by \cite[Proposition 3.11]{GKP22}.
Moreover, since $\det\mathcal Q \cong (\mathcal M^{\otimes n})^{\vee\vee}$ is a $\Q$-line bundle, $\mathcal M$ is also a $\Q$-line bundle. 
$c_1(\mathcal M)=\frac{\delta'(X)}{n+1}c_1(X)$ follows from
 \eqref{eq-Miyaoka-Yau-eq-2} and assertion $(\alpha)$.
\end{proof}
\end{step}

\begin{step}[Proof of Theorem \ref{thm-MY-delta} (2)]
Suppose that $X$ is smooth and equality holds in \eqref{eq-MY-delta-inequality}. We show that
$X \cong \C\mathbb{P}^n$.
Since $X$ is smooth, purity of the branch locus implies that any finite quasi-\'etale cover $\widehat{X} \to X$ is \'etale, and the simple connectedness of $X$ then implies that it is an isomorphism. Thus, by Step \ref{step-MY-equality-2}, it suffices to deduce $X \cong \C\mathbb{P}^n$ from Claim \ref{claim-str-HN-filtration}.

Since $X$ is smooth, the $\Q$-line bundle $\mathcal{M}$ in Claim \ref{claim-str-HN-filtration} is a line bundle.
Since $\frac{n+1}{\delta'(X)}\ge n+1$, the assertion $(\gamma)$ gives $(-K_X)\cdot C \ge n+1$ for every irreducible curve $C \subset X$, and \cite[Corollary 0.3]{CMSB02} implies that $X \cong \C\mathbb{P}^n$.
\end{step}

\begin{step}[Proof of Theorem \ref{thm-MY-delta} (1)]
It suffices to show that Claim \ref{claim-str-HN-filtration} implies $X \cong \mathbb{P}(1,\ldots, 1, p)$.
For the $\Q$-line bundles $\mathcal{F}$, $\mathcal{M}$ in Claim \ref{claim-str-HN-filtration}, choose representations
$\mathcal{F} \cong \mathcal{O}_{X}(F)$, $ \mathcal{M} \cong \mathcal{O}_{X}(M)$
with $\Q$-Cartier Weil divisors $F, M$.  
Since $X$ is a klt Fano variety, we have 
\begin{equation}
\label{eq-defn-FM}
F \underalign{\text{(Claim \ref{claim-str-HN-filtration} $(\alpha)$)}}{\sim_{\Q}} 
\frac{n+1-n \delta'(X)}{n+1}(-K_X)
\quad \text{and} \quad 
M \underalign{\text{(Claim \ref{claim-str-HN-filtration} $(\gamma)$)}}{\sim_{\Q}} 
\frac{ \delta'(X)}{n+1}(-K_X)
\end{equation}
\begin{claim}
\label{claim-global-split-V}
The exact sequence \(
0 \to \mathcal{F} \overset{}{\to} \mathscr{V}\overset{p}{\to} \mathcal{Q} \to 0
\)
splits globally, and $ \mathscr{V} \cong  \mathcal{F}  \oplus \mathcal{M}^{\oplus n}$.
\end{claim}

\begin{proof}[Proof of Claim \ref{claim-global-split-V}]
Since $X$ is projective,
the GAGA comparison theorem for morphisms \cite[Theorem~2]{Ser56} implies that
the above exact sequence splits globally
if and only if its analytification splits globally.
We therefore prove that the analytified exact sequence splits.
To simplify notation, we use the same symbols for $X$, the sheaves, and the morphisms
and for their analytifications.

By the left exactness of the sheaf Hom functor $\Hom(\mathcal{Q}, \cdot)$ and the assumption that the sequence splits analytically locally, we obtain the following exact sequence:
$$
0 \to \Hom(\mathcal{Q},\mathcal{F}) \to \Hom(\mathcal{Q},\mathscr{V})
\overset{p^{*}}{\to}\Hom(\mathcal{Q},\mathcal{Q}) \to 0
$$
where $p^{*}$ is given by composition with $p$. 
Thus we obtain the following exact sequence in sheaf cohomology:
$$
H^0(X, \Hom(\mathcal{Q},\mathscr{V}))
\overset{p_0^{*}}{\to}
H^0(X, \Hom(\mathcal{Q},\mathcal{Q}) )
\to 
H^1(X, \Hom(\mathcal{Q},\mathcal{F}))
$$
If $p_0^{*}$ is surjective, then $\id_{\mathcal Q}$ lifts, and hence the original exact sequence splits globally.
To prove this, it suffices to show that $H^1(X, \Hom(\mathcal{Q},\mathcal{F}))=0$, which follows immediately from the Kawamata--Viehweg vanishing theorem, since  $
\Hom(\mathcal{Q},\mathcal{F})
\cong \mathcal{O}_{X}(F-M)^{\oplus n}
$ and
$(F-M) - K_X \sim_{\Q}
(2- \delta'(X))(-K_X)
$
is ample by \eqref{eq-defn-FM}.
\end{proof}

Now write $\frac{p}{q} = \frac{n+1 - n \delta'(X)}{\delta'(X)}$ with coprime positive integers $p, q$, 
choose integers $a,b$ such that $ap + bq =1$, and define the $\Q$-Cartier divisor 
$A \coloneqq aF + bM$.
By \eqref{eq-defn-FM}, we have
\begin{equation}
\label{eq-defn-Amplepq}
F \sim_\Q pA,\quad
M \sim_\Q qA,\quad
qF - pM \sim_\Q 0, \quad
-K_X \sim_\Q (p+nq)A,\quad -K_X \sim F + nM
\end{equation}
In particular, $A$ is ample. 
We take a positive integer $\ell$ such that $\ell qF\sim \ell pqA \sim \ell pM$ and
$\ell A$ is Cartier. Fix isomorphisms of line bundles given by these linear equivalences.

Claim \ref{claim-global-split-V} and the definition of the canonical extension sheaf give
an analytically locally split exact sequence
\begin{equation}
\label{eq-locally-split-FM}
0 \to  \mathcal{O}_X 
\overset{s}{\to} \mathcal{O}_{X}(F) \oplus \mathcal{O}_{X}(M)^{\oplus n} \to \mathcal{T}_X \to 0
\end{equation}
We write its first map
as
\[
 s=(s_0,\ldots,s_n),
 \qquad
 s_0\in H^0(X,\mathcal O_X(F)) \quad s_1, \ldots, s_n \in H^0(X,\mathcal O_X(M))
\]
\begin{claim}
\label{claim-toric}
$X$ is a toric variety. More precisely, each $s_i$ is nonzero, and $(X, D_0 + \cdots + D_n)$ is a toric pair, where $D_i$ is the effective Weil divisor defined by $s_i$.
\end{claim}
\begin{proof}[Proof of Claim \ref{claim-toric}]
We first show that each $s_i$ is nonzero and hence defines an effective Weil divisor $D_i$.
Let $s_i^{[m]}$ denote the image of $s_i^{\otimes m}$ in the reflexive tensor power. Using the fixed isomorphisms of line bundles, set
\[
t_0\coloneqq s_0^{[\ell q]}
\in H^0\bigl(X,\mathcal O_X(\ell pqA)\bigr),
\qquad
t_i\coloneqq s_i^{[\ell p]}
\in H^0\bigl(X,\mathcal O_X(\ell pqA)\bigr)
\quad (1\le i\le n)
\]
Then,
by the analytic local splitting of \eqref{eq-locally-split-FM},
$t_0,\ldots,t_n$ have no common zero. 
We therefore obtain
\[
f=[t_0:\cdots:t_n]\colon X\longrightarrow\mathbb P^n
\quad
f^*\mathcal O_{\mathbb P^n}(1)
\simeq
\mathcal O_X(\ell pqA)
\]
This implies that $f$ is finite and surjective. Thus each $t_i$ and each $s_i$ is nonzero. 
Since $f$ is finite, $D_i$ and $D_j$ have no common prime components if $i \neq j$.

Set $D\coloneqq D_0 + \cdots + D_n$. (At this stage, we do not yet know whether $D_i$ is reduced or prime.)
By \eqref{eq-defn-Amplepq}, we have $K_X +  D \sim_\Q 0$. 
We show that $(X, D)$ is a log canonical pair. 
Let $H_i\subset\mathbb P^n$ be the $i$-th coordinate hyperplane, and set
\(
B_{\mathbb P}\coloneqq \sum_{i=0}^{n}H_i
\).
Let $R_f$ be the ramification divisor of $f$. It is effective, and
\(
K_X=f^*K_{\mathbb P^n}+R_f
\)
holds. 
Then the ramification formula shows that $D - (f^*B_{\mathbb P}-R_f)$ is effective. 
We also have
$$
D - (f^*B_{\mathbb P}-R_f)
\sim_{\Q} 
-K_X + K_X - f^*(K_{\mathbb P^n}+B_{\mathbb P})
\sim 0.
$$
Hence $D = f^*B_{\mathbb P}-R_f$, 
and finally \(K_X+D=f^*(K_{\mathbb P^n}+B_{\mathbb P})\).
Thus 
$(X,D)$ is log canonical by \cite[Proposition~5.20(4)]{KM98}. Comparing coefficients shows that each $D_i$ is reduced.

To prove that $(X, D)$ is a toric pair, we construct a  decomposition of $D$ in the sense of \cite[Definition~1.1]{BMSZ18} satisfying the assumptions of  \cite[Theorem~1.2]{BMSZ18}.
Consider the decomposition
\[
\sum_{i=0}^n 1\cdot D_i=D
\]
and compute its complexity $c$ as in \cite[Definition~1.1]{BMSZ18}. 
The vector space spanned by $D_0,\ldots,D_n$ modulo algebraic equivalence has dimension $r=1$, since \eqref{eq-defn-Amplepq} gives the relations \(
D_0\sim_{\mathbb Q}pA,
D_i\sim_{\mathbb Q}qA
\) for $1\leq i\leq n$. 
Thus the complexity of this decomposition is $c=n+r-d=0$ since  \(d\coloneqq \sum_{i=0}^n 1=n+1.\)

Hence \cite[Theorem~1.2]{BMSZ18} yields a divisor $B$ such that $(X,B)$ is a toric pair and \(
B\geq D.
\)
(Here we use the fact that $D$ is reduced.)
Moreover, since the complexity is $0$, the second assertion in \cite[Theorem~1.2]{BMSZ18}  implies that every prime component of $B$ is equal to one of $D_0,\ldots,D_n$.
Thus any $D_i$ is prime and $B=D$. 
Thus
\(
(X,D)
\)
is a toric pair.
\end{proof}

It remains to show that $q=1$ and $X\simeq\mathbb P(p,1,\ldots,1)$. Let $T$ be the dense algebraic torus of $X$. Denote its lattice of one-parameter subgroups by $N\coloneqq\operatorname{Hom}(\mathbb G_m,T)$ and its character lattice by $\Lambda\coloneqq \operatorname{Hom}_{\mathbb Z}(N,\mathbb Z)$. For each torus-invariant prime divisor $D_i$, let $\rho_i$ be the corresponding ray in the fan of $X$ and let $v_i\in N$ be its primitive generator.
Since $v_0$ is primitive, we can choose $u\in\Lambda$ with $\langle u,v_0\rangle=1$. The corresponding character $\chi^u \colon T \to \mathbb{G}_m$ then satisfies
\[
 \operatorname{div}(\chi^u)
 =
 D_0+\sum_{i=1}^n\langle u,v_i\rangle D_i
\]
Since $D_0\sim F$ and $D_i\sim M$, it follows that $F\sim bM$ for some integer $b$. By \eqref{eq-defn-Amplepq}, we obtain \(pA \sim_{\mathbb Q}bqA\), hence $p=bq$. Since $p$ and $q$ are coprime, we conclude that $q=1$.

Since  the divisor class group $\operatorname{Cl}(X)$ is generated by $[M]$, 
\cite[Theorem~4.1.3]{CLS10} gives the following exact sequence
\[
 0\longrightarrow\Lambda
 \xrightarrow{u\mapsto  \sum\langle u,v_i\rangle D_i}
 \bigoplus_{i=0}^{n} \Z D_i
 \xrightarrow{ \sum a_i D_i\mapsto (pa_0+a_1+\cdots+a_n)[M]}
 \Z[M]
 \longrightarrow0
\]
Taking the dual of this exact sequence, we obtain
\begin{equation*}
\label{wpe-toric-lattice}
 N\simeq\mathbb Z^{n+1}/\mathbb Z(p,1,\ldots,1),
 \qquad
 pv_0+v_1+\cdots+v_n=0
\end{equation*}
Thus the fan  of the toric variety $X$ agrees with
the standard fan of $\mathbb P(p,1,\ldots,1)$. Consequently,
\(
X\simeq\mathbb P(p,1,\ldots,1).
\)
\end{step}
Combining the above steps completes the proof.
\end{proof}


\begin{rem}
Claim \ref{claim-global-split-V} also directly implies that $X$ has only quotient singularities.
Indeed, for $x \in X$, take a sufficiently small analytic open neighborhood $x \in U \subset X$.
The canonical extension splits on this neighborhood, so
$$
\left(\mathcal{O}_{X} \oplus \mathcal{T}_X\right)\rvert_{U}
\cong 
\mathscr{V}\rvert_{U}
\cong 
\left(\mathcal{O}_{X}(F)\oplus \mathcal{O}_{X}(M)^{\oplus n}\right)\rvert_{U}
$$
Since $F, M$ are $\Q$-Cartier,
after shrinking $U$ and applying the index-one cover construction of \cite[Definition 2.52]{KM98} to $F, M$,
we can take a finite quasi-\'etale Galois cover $\pi \colon \widehat{U} \to U$ such that
$\pi^{[*]}(\mathcal{O}_{X}(F)|_U) \cong \mathcal{O}_{\hat U} \cong \pi^{[*]}(\mathcal{O}_{X}(M)|_U)$.
Since $\widehat{U}$ is klt and
$\mathcal{O}_{\widehat{U}} \oplus \mathcal{T}_{\widehat{U}}
\cong \mathcal{O}_{\widehat{U}}^{\oplus n+1}$,  $\mathcal{T}_{\widehat{U}}$ is locally free. By the Lipman--Zariski theorem for klt spaces \cite[Theorem 6.1]{GKKP11}, $\widehat{U}$ is smooth.
Setting $G\coloneqq\operatorname{Gal}(\widehat U/U)$, we obtain $\widehat{U}/G \cong U$. Hence $X$ has only quotient singularities.

Moreover, Bryant \cite[Theorem 11]{Bry01} showed that every weighted projective space admits a Bochner--K\"ahler orbifold metric.
An alternative differential-geometric proof of the argument after Claim \ref{claim-global-split-V}, related to \cite{Bry01}, may also be possible. See also Remark \ref{rem-Bochner-vanishing}.
\end{rem}

\subsection{Examples and sharpness of the Miyaoka--Yau inequality} 
\label{subsec-examples-sharpness}
To avoid repetition, for an $n$-dimensional klt Fano variety $X$, we define the Miyaoka--Yau constant by
$$
\MY_n(X) \coloneqq \left( 2(n+1)\widehat{c}_2(X) - n c_1(X)^2 \right)  \cdot c_1(X)^{n-2}.
$$
Throughout this subsection, $\delta(X)$ denotes the delta invariant, and we set $\delta'(X) \coloneqq \min \{1, \delta(X)\}$.

In this subsection, we verify that examples for which the Miyaoka--Yau inequality fails, that is, $\MY_n(X)<0$, nevertheless satisfy
\begin{equation}
\label{eq-equality-MY-delta}
\MY_n(X) \ge -n(1-\delta'(X))^2 \cdot c_1(X)^n.
\end{equation}
We also provide examples for which equality holds in \eqref{eq-equality-MY-delta} while $\MY_n(X)<0$.

 We first observe that every smooth Fano variety of dimension at most three satisfies $\MY_n(X) \ge 0$.
\begin{ex}
Suppose that $X$ is a smooth surface such that $-K_X$ is ample. Then $c_1(X)^2 \le 9$ and $\chi(X, \mathcal{O}_X)=1$, and hence
$$
\MY_2(X)
\underalign{\text{(Noether formula)}}{=}
72 \chi(X, \mathcal{O}_X) - 8 c_1(X)^2 \ge 0.
$$
If $X$ is a smooth Fano threefold, then the classification gives $c_1(X)^3 \le 64$,  and hence
$$
\MY_3(X)
\underalign{\text{(Riemann--Roch)}}{=}
192 \chi(X, \mathcal{O}_X) - 3 c_1(X)^3 
\ge 0.
$$
\end{ex}

The first example of a smooth Fano manifold with $\MY_n(X)<0$ appears in dimension four.
\begin{ex}[{cf. \cite[Example 7.1]{GKP22}}]
\label{ex-counter-example-MY}
Let $X \coloneqq \mathbb{P}(\mathcal{O}_{\C\mathbb{P}^3} \oplus \mathcal{O}_{\C\mathbb{P}^3}(3))$ be the projective bundle over $\C\mathbb{P}^3$. 
Since $X$ is a toric variety, a direct computation gives $c_2(X) \cdot c_1(X)^2=296$ and $c_1(X)^4=800$. 
On the other hand, by applying the formula in \cite[Theorem 1.1]{ZZ22} with $V=\C\mathbb{P}^3$, $L=\mathcal{O}_{\C\mathbb{P}^3}(3)$, $r=\frac{4}{3}$, and $\dim V=3$, we obtain $\delta(X)=\delta'(X)=\frac{500}{767}$.
Therefore,
$$
\MY_{4}(X)=-240
\quad \text{and} \quad
-4\left(1-\delta'(X)\right)^2 \cdot c_1(X)^4
=-\frac{228124800}{588289}
\approx -387.7767.
$$
Thus, although $X$ does not satisfy the Miyaoka--Yau inequality, it does satisfy \eqref{eq-equality-MY-delta}.
\end{ex}

\begin{ex}
\label{ex-weighted-projective}
Let \(X \coloneqq \mathbb{P}(a_0,a_1,\ldots,a_n)\) be an \(n\)-dimensional weighted projective space. Assume that \(1\leq a_0\leq a_1\leq\cdots\leq a_n\) are integers and that
\(
\gcd(a_0,\ldots,\widehat{a_i},\ldots,a_n)=1
\)
for every \(0\leq i\leq n\).
Set
$$
S \coloneqq \sum_{i=0}^{n} a_i
\quad \text{and} \quad
P=\prod_{i=0}^{n} a_i.
$$
Notice that $X$ is a $\Q$-factorial klt toric Fano variety. 
Thus the orbifold second Chern class is well-defined by \cite[Theorem 3.13]{GKPT19b}. A direct computation gives
\begin{equation}
\label{eq-weighted-MY}
c_1(X)^n= \frac{S^n}{P} 
\quad\text{and}\quad 
\widehat{c}_2(X)\cdot c_1(X)^{n-2} = \frac{\left(\sum_{0 \le i < j \le n} a_ia_j \right)S^{n-2}}{P}. 
\end{equation}
On the other hand, the formula in \cite[Corollary~7.16]{BJ20} implies
\begin{equation}
\label{eq-weighted--delta}
\delta(X)=\delta'(X)= \frac{(n+1)}{S}a_{0}. 
\end{equation}
From \eqref{eq-weighted-MY} and \eqref{eq-weighted--delta}, it follows that
\begin{equation}
\label{eq-weighted-projective-MY-delta}
\MY_{n}(X)=
- \frac{S^{n-2}}{P}
\sum_{0 \le i < j \le n}(a_i - a_j)^2
\quad\text{and}\quad 
-n \left(1-\delta'(X)\right)^2 \cdot c_1(X)^n
=
-\frac{n S^{n-2}}{P}
(S - (n+1)a_0)^2.
\end{equation}

In particular, if $a_0=\cdots=a_{n-1}=1$ and $a_n=b$ with $b>1$, then 
$$
\MY_{n}(X)=
-\frac{n(n+b)^{n-2}(b-1)^2}{b}
=
-n \left(1-\delta'(X)\right)^2 \cdot c_1(X)^n.
$$
Thus $\mathbb{P}(1, 1,\ldots, 1, b)$ does not satisfy the Miyaoka--Yau inequality, and equality holds in \eqref{eq-equality-MY-delta}.

For $X=\mathbb{P}(a_0, a_1,\ldots, a_n)$, the inequality in Theorem \ref{thm-main-klt} can be proved directly.
Indeed, for non-negative real numbers $b_0, b_1, \ldots, b_n$, the following inequality holds:
\begin{equation}
\label{eq-identity-weighted}
\sum_{0 \le i < j \le n} (b_i - b_j)^2 
= 
\frac{1}{2}\sum_{i=0}^{n} \sum_{j =0}^{n} (b_i - b_j)^2
=
\underbrace{(n+1)\sum_{i=0}^{n} b_i^{2}}
_{\le (n+1) \left(\sum_{i=0}^{n} b_i\right)^2}
- \left(\sum_{i=0}^{n} b_i\right)^2
\le n \left(\sum_{i=0}^{n} b_i\right)^2.
\end{equation}
Substituting $b_i=a_i-a_0$ gives
$$
\MY_{n}(X)
\underalign{\eqref{eq-weighted-projective-MY-delta}}{=}
- \frac{S^{n-2}}{P}
\sum_{0 \le i < j \le n}(b_i - b_j)^2
\underalign{\eqref{eq-identity-weighted}}{\ge}
-n \frac{S^{n-2}}{P}
\left(\sum_{i=0}^{n} b_i\right)^2
\underalign{\eqref{eq-weighted-projective-MY-delta}}{=}
-n \left(1-\delta'(X)\right)^2 \cdot c_1(X)^n.
$$

\end{ex}
In view of the above example, we are also interested in whether Theorem \ref{thm-main-klt} admits a combinatorial proof in the toric case.

\subsection{The weak Fano case}
\label{subsec-weak-fano}
In this subsection, we extend Theorem \ref{thm-main-klt}
to projective klt weak Fano varieties.
We first compare the orbifold second Chern numbers under
crepant birational morphisms.

\begin{lem}[{cf. \cite[Theorem 4.2]{Meg99}}]
\label{lem:crepant-surface-chern}
Let $g\colon S\to T$ be a crepant birational morphism
between projective klt surfaces.
Then
\[
\widehat{c}_2(S)\ge\widehat{c}_2(T).
\]
Moreover, equality holds if and only if $g$ is an isomorphism.
\end{lem}

\begin{proof}
Suppose that $g$ is not an isomorphism.
We factor $g$ as
\[
S=S_0
\xrightarrow{g_0}S_1
\xrightarrow{g_1}\cdots
\xrightarrow{g_{r-1}}S_r=T,
\]
where each $g_i$ is a projective birational morphism
contracting a single irreducible curve;
cf. \cite[proof of Theorem 0.1(ii)]{Meg99}.

Since each $S_i$ is klt and each $g_i$ is crepant, applying \cite[Theorem 4.2]{Meg99} with zero boundary gives
\[
\lambda(g_i)
\coloneqq
K_{S_i}^2-4\widehat{c}_2(S_i)
-K_{S_{i+1}}^2+4\widehat{c}_2(S_{i+1})
\le0.
\]
In fact, we know $\lambda(g_i)<0$. Indeed, the proof of \cite[Theorem 4.2]{Meg99}, as summarized in \cite[p.~274]{Meg99}, shows that equality requires the
contracted curve to have coefficient $1$ in the boundary,
whereas our boundary is zero.
Therefore,
\[
\begin{aligned}
0
&>
\sum_{i=0}^{r-1}\lambda(g_i)=
\underbrace{K_S^2-K_T^2}_{=0}
-4\bigl(\widehat{c}_2(S)-\widehat{c}_2(T)\bigr).
\end{aligned}
\]
Thus $\widehat{c}_2(S)>\widehat{c}_2(T)$. 
If $g$ is an isomorphism, then $\widehat{c}_2(S)=\widehat{c}_2(T)$,
so the assertion follows.
\end{proof}

The following corollary is essentially proved in
\cite[proof of Claim 7.1]{GKPT20}.
\begin{cor}
\label{cor:crepant-chern-comparison}
Let $X$ and $Y$ be projective klt varieties of dimension $n\ge2$,
and let $f\colon X\to Y$ be a crepant birational morphism.
Let $H$ be an ample $\Q$-Cartier divisor on $Y$.
Then
\[
C_f(H)
\coloneqq
\widehat{c}_2(X)\cdot(f^*H)^{n-2}
-
\widehat{c}_2(Y)\cdot H^{n-2}
\ge0.
\]
Moreover, $C_f(H)=0$ if and only if there exists a closed subset
$Z\subset Y$ with $\operatorname{codim}_Y Z\ge3$ such that
\(
X \setminus f^{-1}(Z)\rightarrow Y\setminus Z
\)
is an isomorphism.
\end{cor}

\begin{proof}
We may  assume that $n\ge3$.
Take a sufficiently divisible positive integer $m$ such that
$mH$ is very ample, and let
\[
S_Y\coloneqq D_1\cap\cdots\cap D_{n-2},
\qquad
S_X\coloneqq f^{-1}(S_Y),
\]
where $D_1,\ldots,D_{n-2}\in|mH|$ are sufficiently general.
By Bertini's theorem, $S_X$ and $S_Y$ are projective klt surfaces
and $g\coloneqq f|_{S_X}\colon S_X\to S_Y$ is a crepant birational
morphism.
As in \cite[proof of Claim 7.1]{GKPT20}, we have exact
sequences of $\Q$-vector bundles
\[
\begin{aligned}
0&\longrightarrow\mathcal{T}_{S_X}
\longrightarrow\mathcal{T}_X|_{S_X}
\longrightarrow N_{S_X/X}\longrightarrow0,
\quad \text{and} \quad 
0&\longrightarrow\mathcal{T}_{S_Y}
\longrightarrow\mathcal{T}_Y|_{S_Y}
\longrightarrow N_{S_Y/Y}\longrightarrow0.
\end{aligned}
\]
This gives the following equalities of orbifold Chern numbers
\[
\begin{aligned}
\widehat{c}_2(\mathcal{T}_X|_{S_X})
&=
\widehat{c}_2(S_X)
-K_{S_X}\cdot c_1(N_{S_X/X})+\widehat{c}_2(N_{S_X/X}),\\
\widehat{c}_2(\mathcal{T}_Y|_{S_Y})
&=
\widehat{c}_2(S_Y)
-K_{S_Y}\cdot c_1(N_{S_Y/Y})+\widehat{c}_2(N_{S_Y/Y}).
\end{aligned}
\]
Using $K_{S_X}=g^*K_{S_Y}$ and $N_{S_X/X}\cong g^*N_{S_Y/Y}$, 
the projection formula yields
\[
\begin{aligned}
&m^{n-2}C_f(H) \\
&=
\underbrace{
\widehat{c}_2(S_X)-\widehat{c}_2(S_Y)
}_{\ge0\ \text{by Lemma \ref{lem:crepant-surface-chern}}}
-\underbrace{
\bigl(K_{S_X}\cdot c_1(N_{S_X/X})
      -K_{S_Y}\cdot c_1(N_{S_Y/Y})\bigr)
}_{=0}
+\underbrace{
\bigl(\widehat{c}_2(N_{S_X/X})-\widehat{c}_2(N_{S_Y/Y})\bigr)
}_{=0}
\ge0.
\end{aligned}
\]

We now prove the assertion concerning equality.
Let $E\subset Y$ be the closed subset over which $f$ is not
an isomorphism.
By Lemma \ref{lem:crepant-surface-chern}, $C_f(H)=0$ if and
only if $g$ is an isomorphism, or equivalently, $S_Y\cap E=\varnothing$.
Since a general $S_Y$ avoids $E$ if and only if
$\operatorname{codim}_Y E\ge3$, this proves the assertion.
\end{proof}

Thus, Theorem \ref{thm-main-klt} extends to the weak Fano case,
that is, when $-K_X$ is nef and big.

\begin{thm}
\label{thm:weak-fano-miyaoka-yau}
Let $X$ be an $n$-dimensional projective klt weak Fano variety
with $n\ge2$.
Then the following Miyaoka--Yau inequality holds:
\begin{equation}
\label{eq:display-006}
\left(2(n+1)\widehat{c}_2(X)-n c_1(X)^2\right)
\cdot c_1(X)^{n-2}
\ge -n(1-\delta'(X))^2\cdot c_1(X)^n.
\end{equation}
Moreover, if equality holds in \eqref{eq:display-006},
then the anticanonical morphism
$\varphi\colon X\to X_{\mathrm{an}}$ satisfies the following:
\begin{enumerate}
\item There exists a closed subset $Z\subset X_{\mathrm{an}}$
with $\operatorname{codim}_{X_{\mathrm{an}}}Z\ge3$ such that
\(
X\setminus\varphi^{-1}(Z)\rightarrow X_{\mathrm{an}}\setminus Z
\)
is an isomorphism.
\item There exist a finite quasi-\'etale cover
$\widehat{X_{\mathrm{an}}}\to X_{\mathrm{an}}$
and a positive integer $p\in\N$ such that
$\widehat{X_{\mathrm{an}}}\cong\mathbb{P}(1,\ldots,1,p)$.
\end{enumerate}
\end{thm}
Note that the delta invariant is defined in \cite[Section 4.2]{BJ20}
even when $-K_X$ is only big.
\begin{proof}
As in Subsection \ref{subsec-examples-sharpness}, we denote the
left-hand side of \eqref{eq:display-006} by $\MY_n(X)$.
Since $-K_X$ is semiample, we obtain the anticanonical morphism
$\varphi\colon X\to X_{\mathrm{an}}$.
Then $X_{\mathrm{an}}$ is a projective klt Fano variety and
\(K_X=\varphi^*K_{X_{\mathrm{an}}}\).
Thus $\varphi$ is crepant, and this gives equality of the log discrepancies, and the expected
vanishing orders agree by \cite[Lemma 3.7(ii)]{BJ20}.
Thus \cite[Theorem 4.4]{BJ20} gives
$\delta'(X)=\delta'(X_{\mathrm{an}})$.
The projection formula also gives $c_1(X)^n=c_1(X_{\mathrm{an}})^n$.
Hence we obtain
\[
\begin{aligned}
\MY_n(X)
&\underalign{}{=}\MY_n(X_{\mathrm{an}})+2(n+1)
C_{\varphi}(-K_{X_{\mathrm{an}}})
\underalign{\text{(Cor.~\ref{cor:crepant-chern-comparison})}}{\ge}
\MY_n(X_{\mathrm{an}})\\
&\underalign{\text{(Thm.~\ref{thm-main-klt})}}{\ge}
-n(1-\delta'(X_{\mathrm{an}}))^2\cdot c_1(X_{\mathrm{an}})^n
=-n(1-\delta'(X))^2\cdot c_1(X)^n.
\end{aligned}
\]
This proves \eqref{eq:display-006}.
If equality holds in \eqref{eq:display-006}, then
$C_{\varphi}(-K_{X_{\mathrm{an}}})=0$ and equality holds in
Theorem \ref{thm-main-klt} for $X_{\mathrm{an}}$.
The assertions therefore follow from
Corollary \ref{cor:crepant-chern-comparison}
and Theorem \ref{thm-main-klt}.
\end{proof}

\begin{ex}
\label{ex-weak-fano-sharp}
Let $X\coloneqq
\mathbb{P}_{\C\mathbb{P}^2}
(\mathcal{O}_{\C\mathbb{P}^2}\oplus\mathcal{O}_{\C\mathbb{P}^2}(3))$
be the projective bundle over $\C\mathbb{P}^2$.
Then $X$ is a smooth weak Fano threefold which is not Fano,
and its anticanonical model is $\mathbb{P}(1,1,1,3)$.
A direct computation gives
$c_1(X)^3=72$, $c_2(X)\cdot c_1(X)=24$ and 
$\delta(X)=\delta'(X)=\frac{2}{3}$.
Therefore,
\[
\MY_3(X)
=-24
=-3(1-\delta'(X))^2\cdot c_1(X)^3.
\]
Thus, although $X$ does not satisfy the Miyaoka--Yau inequality,
equality holds in \eqref{eq:display-006}.
\end{ex}

\section{Relation of the Miyaoka--Yau inequality to Donaldson's conjecture}
\label{sec-Calabi}

In this section, we give another approach to Theorem 
\ref{thm-main-klt}, 
assuming that $X$ is smooth.  
This is based on Donaldson's conjecture for the 
lower bound of the Calabi functional 
  \begin{equation}\label{Calabi functional}
  \Cal(\omega)=\bigg[ \frac{1}{V}\int_X (S(\omega)-\widehat{S})^2 \omega^n
  \bigg]^{\frac{1}{2}}. 
  \end{equation}
  Here we fix a polarized manifold $(X, L)$ 
  and a smooth K\"ahler metric $\omega \in c_1(L)$. 
  $S(\omega)=\tr_{\omega} \Ric(\omega)$ is the scalar curvature. The volume $V=\int_X \omega^n$ and 
   the mean value of the scalar curvature $\widehat{S} \coloneqq \frac{1}{V}\int_X S(\omega)\omega^n$ are numerical constants. In particular, if one takes 
   $L=-K_X$ with $X$ Fano, $\widehat{S}=n$. 
In \cite{Don05}, Donaldson predicted that the lower bound of $\Cal(\omega)$ is given in terms of his generalization of the  
classical Futaki invariant. 

The picture becomes clearer if one uses the non-Archimedean terminology introduced in 
\cite{BHJ17}.
For the identification 
  of a given test configuration 
  with the non-Archimedean metric, 
see \cite[Definition $6.1$]{BHJ17} and the exposition below it. 
More precisely, for any smooth non-Archimedean metric $\Phi$
on the Berkovich analytification of $X$, 
one can define the value of the non-Archimedean K-energy $M^{\NA}(\Phi)$, 
following \cite[Definition $7.13$]{BHJ17}.  
This is the non-Archimedean counterpart of Mabuchi's K-energy functional defined on the space of K\"ahler metrics 
\begin{equation*}
\cH \coloneqq  \bigg\{ \phi \in C^\infty(X, \R): \omega_0 + \ddbar \phi >0\bigg\}.  
\end{equation*}
In the notation, it is convenient to fix a reference $\omega_0$ and identify $\omega \in c_1(L)$ with 
a K\"ahler potential $\phi$ which is 
unique up to an additive constant. Throughout this section,
$dd^c=\ddbar$ and
$\Ric(\omega)=-dd^c\log\omega^n\in c_1(X)$.  
Similarly, we denote the space of smooth non-Archimedean metrics 
by $\cH^{\NA}$. Here these metrics are represented by normal ample
test configurations $(\cX,\cL)$, with rational polarizations allowed.
We write $M^{\NA}(\cX,\cL)=M^{\NA}(\Phi)$. 
See \cite[Definition $6.2$]{BHJ17}
for the precise explanation. 
A polarized manifold is called K-semistable 
if $M^{\NA}(\cX, \cL) \geq 0$ holds 
for all test configurations. 
Otherwise, it is called K-unstable.

Also, in \cite[Definition $6.5$]{BHJ17}, 
 the $L^p$-norm 
 $\norm{\Phi}_p$ was studied for any $p \geq 1$. 
 This is the non-Archimedean counterpart of the infinitesimal length of geodesics on $\cH$, 
 with respect to the canonical Finsler metric 
 \begin{equation*}
 u \mapsto
 \bigg[ \frac{1}{V}\int_X \abs{u}^p (\omega_0 +\ddbar \phi)^n
 \bigg]^{\frac{1}{p}}
 \end{equation*}
 defined at each point $\phi \in \cH$. 
This is equivalent to \cite[Theorem $1.2$]{His16}. 
Using this terminology, one can state Donaldson's conjecture as follows. 

\begin{conj}[\cite{Don05}]\label{Donaldson's conjecture}
  For any K-unstable polarized manifold $(X, L)$ 
  one gets the equality 
\begin{equation}\label{eq-Donaldson-conjecture}
  \inf_{\omega \in \cH} \Cal(\omega)
  = \sup_{\substack{\Phi \in \cH^{\NA}}}
\frac{-M^{\NA}(\Phi)}{\norm{\Phi}_2}. 
\end{equation}
\end{conj}
Here the supremum is taken for all test configurations with non-zero norm. 
The one-sided inequality $\geq$ was proved by Donaldson. 
A weak version which replaces 
$\cH$ with the metric completion $\cE^2$ and 
$\cH^{\NA}$ with the equivalence class of weak geodesic rays on $\cE^2$ was proved by 
\cite{Xia21}. See also G. Sz\'ekelyhidi's thesis for the toric case. 
If one replaces $\Cal(\omega)$ with the Ricci-Calabi functional, 
the analogue of Conjecture \ref{Donaldson's conjecture} was proved in \cite{CHT17}, \cite{His23}. 

On the other hand, in terms of the minimum norm $\norm{\Phi}_{\min}$ introduced by \cite{Der16}, 
  one may link the non-Archimedean K-energy 
  with the delta invariant. 
  
  \begin{thm}[\cite{BLZ22}, \cite{Fuj19}, \cite{LXZ22}]
  \label{thm-M^NA-delta}
   For any K-unstable anti-canonically polarized manifold, the following equality holds:
  \begin{equation}\label{eq-BLZ22}
    \sup_{\substack{\Phi \in \cH^{\NA}}}\frac{-M^{\NA}(\Phi)}{\norm{\Phi}_{\min}}
    =1-\delta(X). 
        \end{equation} 
  Moreover, the supremum is 
  achieved by a special test configuration. 
  Here a representative $(\cX, \cL)$ of $\Phi$ is called special if it 
  satisfies $\cL \sim_{\Q} -K_{\cX/\mathbb A^1}$ and
  $\cX_0$ is a klt $\Q$-Fano variety.   
      \end{thm}
 For details of the above theorem, see \cite[Theorem $1.1$]{BLZ22} and postscript remarks therein. See also \cite[Definition $3.5$ and Theorem $3.6$]{BJ23}. 

Now we will show the relation between  
Donaldson's conjecture 
and the Miyaoka--Yau inequality. 
\begin{thm}\label{thm-Donaldson-MY}
    Let $X$ be an $n$-dimensional 
    Fano manifold with $n\ge2$ and take the anti-canonical polarization $L=-K_X$. 
    Let us assume Conjecture \ref{Donaldson's conjecture}.
     Then the following inequality holds:
      \begin{equation}\label{eq-weak-MY}
       \left( 2(n+1)c_2(X) -n c_1(X)^2 \right) \cdot c_1(X)^{n-2} \geq 
        -n (1-\delta'(X))^2 \cdot c_1(X)^n, 
         \end{equation} 
where $\delta'(X) \coloneqq \min\{1, \delta(X)\}$ 
is the greatest Ricci lower bound. 
\end{thm} 

The above inequality follows immediately from Theorem \ref{thm-main-klt}. Nevertheless, we include this argument because it appears to provide a different perspective on the Miyaoka--Yau inequality. 
The rest of this section is 
devoted to the proof of Theorem \ref{thm-Donaldson-MY}. 

Let us start by recalling some basic facts about the 
Riemannian curvature computations.  
We use the same notation as in Subsection \ref{subsubsec-review-DG}. Let $\omega$ be a   K\"ahler metric on an $n$-dimensional K\"ahler manifold $X$. We denote its Riemann curvature tensor $\Rm(\omega)$ and scalar curvature $S(\omega)$ by
\begin{equation*}
\Rm(\omega)_{i \bar{j}k \bar{l}}
 \coloneqq 
R_{i \bar{j}k \bar{l}}
\quad \text{and} \quad
S(\omega) \coloneqq \tr_{\omega}\Ric(\omega)
=
\frac{1}{2 \pi}g^{i \bar{j}}R_{i\bar{j}},
\end{equation*}
where locally
$
\omega=\sqrt{-1}g_{i \bar{j}}dz^i\wedge d\bar{z}^{\bar{j}}.
$
Thus the components of the Ricci form satisfy
\begin{equation*}
\Ric(\omega)=\sqrt{-1}\Ric(\omega)_{i\bar j}\,dz^i\wedge d\bar z^{\bar j},
\qquad \Ric(\omega)_{i\bar j}=\frac{1}{2\pi}R_{i\bar j}.
\end{equation*}
We also define the normalized curvature tensors by
\begin{equation}
\label{eq-defn-normalized-ricci}
\begin{aligned}
&\widetilde{\Rm}(\omega)_{i\bar jk\bar l}
\coloneqq\frac{1}{2\pi}\Rm(\omega)_{i\bar jk\bar l}
-\frac{S(\omega)}{n(n+1)}(g_{i\bar j}g_{k\bar l}+g_{i\bar l}g_{k\bar j}),\\
&\widetilde{\Ric}(\omega)_{i\bar j}
\coloneqq\Ric(\omega)_{i\bar j}-\frac{S(\omega)}{n}g_{i\bar j}.
\end{aligned}
\end{equation}
The factor $\frac{1}{2\pi}$ appears because we use the convention $\{\Ric(\omega)\}=c_1(X)$.
With this notation, we have the following fundamental identity due to \cite{CO75, Yau77}.
\begin{prop}[{cf. \cite[proof of Proposition]{CO75}}]
\label{prop-fundamental-identity}
For an $n$-dimensional compact K\"ahler manifold $(X,\omega)$, 
\begin{align*}
&\left(2(n+1)c_2(X)-n c_1(X)^2\right)\{\omega\}^{n-2}\\
&\qquad=
\frac{1}{n(n-1)}
\int_X\left((n+1)
\abs{\widetilde{\Rm}(\omega)}_{\omega}^2
-(n+2)\abs{\widetilde{\Ric}(\omega)}_{\omega}^2\right)\omega^n.
\end{align*}
\end{prop}

We now prove Theorem \ref{thm-Donaldson-MY}.
If $X$ is K-semistable, then \eqref{eq-weak-MY} is already known by \cite{Bando, SW16, GKP22, DGP24}. We may therefore assume that $\delta(X)<1$. 
Write the Miyaoka--Yau constant as 
\begin{equation*}
\MY_{n}(X) \coloneqq \left(2(n+1)c_2(X)-n c_1(X)^2\right)\cdot c_1(X)^{n-2}.
\end{equation*}

To obtain the desired estimate, we need a more detailed analysis of the term $\widetilde{\Rm}(\omega)$. For $\omega\in\mathcal{H}$, set
\begin{equation*}
\Ric (\omega)\mathbin{\odot}g
\coloneqq
\frac{1}{2\pi}\bigl(R_{i\bar{j}}g_{k\bar{l}}
+R_{i\bar{l}}g_{k\bar{j}}
+g_{i\bar{j}}R_{k\bar{l}}
+g_{i\bar{l}}R_{k\bar{j}}\bigr),
\end{equation*}
and define
\begin{equation*}
\begin{aligned}
(\widetilde{\Ric}(\omega)\mathbin{\odot}g)_{i\bar jk\bar l}
\coloneqq\widetilde{\Ric}(\omega)_{i\bar j}g_{k\bar l}
+\widetilde{\Ric}(\omega)_{i\bar l}g_{k\bar j}
\quad+g_{i\bar j}\widetilde{\Ric}(\omega)_{k\bar l}
+g_{i\bar l}\widetilde{\Ric}(\omega)_{k\bar j}.
\end{aligned}
\end{equation*}
A direct computation gives
\begin{equation}
\label{eq-wideric-ric}
\begin{aligned}
&\widetilde{\Ric}(\omega)\mathbin{\odot}g
=\Ric(\omega)\mathbin{\odot}g
-\frac{2S(\omega)}{n}(g_{i\bar j}g_{k\bar l}+g_{i\bar l}g_{k\bar j}),\\
&\abs{\widetilde{\Ric}(\omega)\mathbin{\odot}g}_\omega^2
=4(n+2)\abs{\widetilde{\Ric}(\omega)}_\omega^2.
\end{aligned}
\end{equation}
Following \cite[(2.9)]{CO75}, define the Bochner curvature tensor by
\begin{equation*}
\begin{aligned}
\mathrm{B}(\omega)_{i\bar{j}k\bar{l}}
& \coloneqq 
\frac{1}{2\pi}\Rm(\omega)_{i \bar{j}k \bar{l}}
-\frac{1}{n+2}\Ric (\omega)\mathbin{\odot}g
+\frac{S(\omega)}{(n+1)(n+2)}
(g_{i\bar{j}}g_{k\bar{l}}+g_{i{\bar{l}}}g_{k\bar{j}}).
\end{aligned}
\end{equation*}
Then $\mathrm{B}(\omega)$ is trace-free and satisfies
$
\mathrm{B}(\omega)_{i\bar{j}k\bar{l}}
= \widetilde{\Rm}(\omega)_{i \bar{j}k \bar{l}}
-\frac{1}{n+2}\widetilde{\Ric}(\omega) \mathbin{\odot}g
$
from \eqref{eq-wideric-ric}.
In particular, $\langle \mathrm{B}(\omega), \widetilde{\Ric}(\omega) \mathbin{\odot}g\rangle_{\omega}=0$, and hence
\begin{equation}
\begin{aligned}
\label{eq-Em-abs-B}
\abs{\widetilde{\Rm}(\omega)}_{\omega}^2
= 
\abs{\mathrm{B}(\omega)
+\frac{1}{n+2}\widetilde{\Ric}(\omega) \mathbin{\odot}g
}_{\omega}^2
\underalign{\eqref{eq-wideric-ric}}{=}
  \abs{\mathrm{B}(\omega)}^{2}_{\omega}
  + \frac{4}{n+2}\abs{\widetilde{\Ric}(\omega) }_{\omega}^{2}.
\end{aligned}
\end{equation}

For $\omega\in\mathcal{H}$, it follows from \eqref{eq-defn-normalized-ricci} that
$
\langle\widetilde{\Ric}(\omega),\omega\rangle_{\omega}=0.
$
Hence
\begin{equation}
\label{eq-normalized-Ricci}
\abs{\Ric(\omega)-\omega}^2_{\omega}
\underalign{\eqref{eq-defn-normalized-ricci}}{=}
\abs{\widetilde{\Ric}(\omega)+\frac{(S(\omega)-n)}{n}\omega}^2_{\omega}
=
\abs{\widetilde{\Ric}(\omega)}_{\omega}^2
+\frac{1}{n}(S(\omega)-n)^2.
\end{equation}

Moreover, since
$
\langle\Ric(\omega),\omega\rangle_{\omega}
=\tr_{\omega}\Ric(\omega)=S(\omega),
$
we obtain
\begin{equation}
\label{eq-Ricci-Calabi}
\begin{aligned}
\int_X\abs{\Ric(\omega)-\omega}_\omega^2\omega^n
&\underalign{}{=}\int_X\bigl(\abs{\Ric(\omega)}_\omega^2-2S(\omega)+n\bigr)\omega^n\\
&\underalign{\raisebox{-0.9ex}{\tiny\text{\cite[Lemma 4.7]{Sze14}}}}{=}
\int_XS(\omega)^2\omega^n
-n(n-1)\int_X\Ric(\omega)^2\wedge\omega^{n-2} -2\int_XS(\omega)\omega^n+nV\\
&\underalign{}{=}\int_X(S(\omega)-n)^2\omega^n
\underalign{\eqref{Calabi functional}}{=}\Cal(\omega)^2\cdot c_1(X)^n.
\end{aligned}
\end{equation}
Here we used $\{\Ric(\omega)\}=\{\omega\}$ and
$\int_X S(\omega)\omega^n=nV$.
Therefore, combining \eqref{eq-normalized-Ricci} and \eqref{eq-Ricci-Calabi},
\begin{equation}
\label{eq-Ricci-Calabi-2}
\begin{aligned}
\int_X\abs{\widetilde{\Ric}(\omega)}_\omega^2\omega^n
&\underalign{\eqref{eq-normalized-Ricci}}{=}
\int_X\abs{\Ric(\omega)-\omega}_\omega^2\omega^n
-\frac1n\int_X(S(\omega)-n)^2\omega^n\\
&\underalign{\eqref{eq-Ricci-Calabi}}{=}
\frac{n-1}{n}\Cal(\omega)^2\cdot c_1(X)^n.
\end{aligned}
\end{equation}
Therefore, it follows that
\begin{equation}
\label{eq-MY-Calabi}
\begin{aligned}
\MY_{n}(X)
&\underalign{\text{(Prop. \ref{prop-fundamental-identity})}}{=}
\frac{1}{n(n-1)}
\int_X\left((n+1)
\abs{\widetilde{\Rm}(\omega)}_{\omega}^2
-(n+2)\abs{\widetilde{\Ric}(\omega)}_{\omega}^2\right)\omega^n
\\
&\underalign{\eqref{eq-Em-abs-B}}{=}
\frac{n+1}{n(n-1)}
\int_X
\abs{\mathrm{B}(\omega)}_{\omega}^2\omega^n
-
\frac{n}{(n-1)(n+2)}
\int_X \abs{\widetilde{\Ric}(\omega)}_{\omega}^2 \omega^n 
\\
&\underalign{\eqref{eq-Ricci-Calabi-2}}{=}
\frac{n+1}{n(n-1)}
\int_X
\abs{\mathrm{B}(\omega)}_{\omega}^2\omega^n
-
\frac{1}{(n+2)}
\Cal(\omega)^2 \cdot c_1(X)^n.
\end{aligned}
\end{equation}

In view of Theorem \ref{thm-M^NA-delta}, 
it is then natural to compare the 
        minimum norm with the $L^2$-norm. 
In \cite[Section 7]{BHJ17},  
the minimum norm was compared with the $L^1$-norm and as a consequence one has 
        \begin{equation*}
          \norm{\Phi}_2 
          \geq \frac{2n^{n-1}}{(n+1)^{n+1}}\norm{\Phi}_{\min}. 
          \end{equation*}
We shall directly compare the minimum norm with the $L^2$-norm.

\begin{claim}\label{claim-minimal-l2}
Let $(X,L)$ be a polarized manifold of dimension $n\geq2$ and
$\Phi\in\cH^{\NA}$. Then
\begin{equation*}
\norm{\Phi}_2\geq\frac{1}{\sqrt{n(n+2)}}\norm{\Phi}_{\min}.
\end{equation*}
\end{claim}

From the above results Theorem \ref{thm-Donaldson-MY} easily follows.
Indeed for $M^{\NA}(\Phi)<0$, \eqref{eq-BLZ22} and Claim \ref{claim-minimal-l2} give
\begin{equation*}
\frac{-M^{\NA}(\Phi)}{\norm{\Phi}_2}
\leq (1-\delta(X))\frac{\norm{\Phi}_{\min}}{\norm{\Phi}_2}
\leq (1-\delta(X))\sqrt{n(n+2)}.
\end{equation*}
Since we assume \eqref{eq-Donaldson-conjecture}, for every
$0<\varepsilon<1$ there exists $\omega_\varepsilon\in\cH$ such that
\begin{equation}\label{eq-Donaldson-conjecture-again}
\Cal(\omega_\varepsilon)
\leq\sup_{\substack{\Phi\in\cH^{\NA}}}
\frac{-M^{\NA}(\Phi)}{\norm{\Phi}_2}+\varepsilon
\leq (1-\delta(X))\sqrt{n(n+2)}+\varepsilon.
\end{equation}
It follows that
\begin{align*}
\MY_n(X)
&\underalign{\eqref{eq-MY-Calabi}}{\geq}
-\frac{\Cal(\omega_\varepsilon)^2}{n+2}\,c_1(X)^n
\underalign{\eqref{eq-Donaldson-conjecture-again}}{\geq}
-\frac{((1-\delta(X))\sqrt{n(n+2)}+\varepsilon)^2}{n+2}\,c_1(X)^n.
\end{align*}
Letting $\varepsilon\to0$, we obtain \eqref{eq-weak-MY}.

\begin{proof}[Proof of Claim \ref{claim-minimal-l2}]
The proof uses the concave transform of a filtered linear series, as in
\cite[Theorem A]{BC11} and \cite[Theorem 1.1]{WN12}, together with
the interpretation of the minimum norm by a general divisor in \cite{Der16},
based on the weights of the induced test configuration on a divisor
in \cite[Lemma 9 and the proof of Theorem 12]{LS15}.
We give the restriction argument
and the required one-variable estimate explicitly.
We may assume that $\cL$ is a line bundle.
Put $V=L^n$, and let $(\overline{\cX},\overline{\cL})$ be the compactification
of $(\cX,\cL)$ obtained using its trivialization away from the origin. 

First take a common smooth equivariant model
\begin{equation*}
p\colon \overline{\cY}\longrightarrow\overline{\cX},\qquad
q\colon \overline{\cY}\longrightarrow X\times\mP^1,
\end{equation*}
which agrees with the trivial family away from $0$.
For a prime component $F_i$ of the fiber over $0$, write
$b_i=\operatorname{ord}_{F_i}(t)$ and
$v_i=b_i^{-1}\operatorname{ord}_{F_i}|_{\C(X)}$.
The intersection formulas
of \cite[Sections 6--7]{BHJ17} give 
the non-Archimedean Monge--Amp\`ere energy as 
\begin{equation}\label{eq-restriction-minimum}
\begin{aligned}
E^{\NA}(\Phi)=\frac{(p^*\overline{\cL})^{n+1}}{(n+1)V}, 
\quad \text{and}\quad
\norm{\Phi}_{\min} 
=E^{\NA}(\Phi)-\frac{(p^*\overline{\cL}-q^*\pr_X^*L)\cdot (p^*\overline{\cL})^n}{V}.
\end{aligned}
\end{equation}

For sufficiently large $m$, take a smooth irreducible divisor
$D\in|mL|$ containing none of the centers of $v_i$. 
Let $\overline{\cD}$ be the closure of $D\times\C^*$ in
$\overline{\cX}$, put $\cD=\overline{\cD}|_{\mathbb A^1}$,
and let $D_{\cY}$ be its strict transform on ${\cY}$.
The total transform of $D\times\mP^1$ satisfies
\begin{equation*}
\begin{gathered}
D_{\cY} \sim_{\Q}mq^*\pr_X^*L-\sum_i b_i v_i(D)F_i=mq^*\pr_X^*L
\quad \text{and}\quad
D_{\cY} \cdot (p^*\overline{\cL})^n=mq^*\pr_X^*L\cdot (p^*\overline{\cL})^n.
\end{gathered}
\end{equation*}
The restriction $(\cD,\cL|_{\cD})$ is an ample test configuration for
$(D,L|_D)$, possibly nonnormal. Indeed, the schematic closure is integral
and torsion-free over $\C[t]$, hence flat. Its polarization is
relatively ample and $(L|_D)^{n-1}=mV$.
By the barycenter formula \cite[Proposition 3.12(iii)]{BHJ17}, which applies
also to this nonnormal test configuration, its 
non-Archimedean Monge--Amp\`ere energy is
\begin{equation*}
E_D=\frac{(\overline{\cL}|_{\overline{\cD}})^n}{nmV}
=\frac{D_Y\cdot (p^*\overline{\cL})^n}{nmV}
=\frac{q^*\pr_X^*L\cdot (p^*\overline{\cL})^n}{nV}.
\end{equation*}
Consequently \eqref{eq-restriction-minimum} gives, in our volume
normalization, the restriction formula for the minimum norm from \cite{Der16}:
\begin{equation}\label{eq-minimum-restriction-mean}
\norm{\Phi}_{\min}=n\bigl(E_D-E^{\NA}(\Phi)\bigr).
\end{equation}

We next express these invariants using Okounkov bodies.
Choose an admissible flag $Y_\bullet$ on $X$ whose first member is $D$;
its subsequent members give an admissible flag on $D$.
The flag valuation $\nu(s)\in\Z^n$ of a nonzero section records
successive orders of vanishing along this flag after division and
restriction at each step. Let $\Delta$ be the closed convex hull of
$k^{-1}\nu(s)$ for $0\ne s\in H^0(X,kL)$, and let $\Delta_D$ be the analogous
body for $L|_D$ with the induced flag. In the coordinates $(t,y)$ with
$t=\nu_1/k$, the slicing theorem \cite[Theorem 4.26]{LM09} gives
\begin{equation}\label{eq-divisor-Okounkov-cone}
\Delta=\bigl\{(t,(1-mt)y):0\leq t\leq1/m,\ y\in \Delta_D\bigr\}.
\end{equation}

A test configuration 
defines a decreasing multiplicative filtration $\cF$:
for an integer $\lambda$, a section $s$ belongs to $\cF^\lambda$ if
$t^{-\lambda}\bar s$ extends over $\cX$, where $\bar s$ is its invariant
extension away from $0$. 
Relative Serre vanishing applied to the ideal of $\cD$ makes
\begin{equation*}
H^0(\cX,k\cL)\longrightarrow H^0(\cD,k\cL\vert_{\cD})
\end{equation*}
surjective for $k$ sufficiently large. Thus the filtration induced by $(\cD, \cL \vert_{\cD})$ is the quotient filtration induced by restriction.
All jumping weights $\lambda$ are linearly bounded in $k$. 

We use the concave-transform theorem of
\cite[Theorem A]{BC11}, applied to test configurations as in
\cite[Theorem 1.1 and Section 6]{WN12}.
A multiplicative, linearly bounded filtration of the complete section
ring of an ample line bundle defines a bounded concave function on its
Okounkov body. 
More explicitly, for $\lambda\in\R$, let
\begin{equation*}
V_k^\lambda
:=\cF^{\lceil k\lambda\rceil}H^0(X,kL),
\qquad
V_\bullet^\lambda:=\bigoplus_{k\geq0}V_k^\lambda.
\end{equation*}
Let $\Delta(V_\bullet^\lambda)$ be the Okounkov body of this
graded linear series with respect to $Y_\bullet$. 
The concave transform is defined by
\begin{equation*}
G(x):=\sup\bigl\{\lambda\in\R:
                 x\in\Delta(V_\bullet^\lambda)\bigr\},
\end{equation*}
for each $x \in \Delta$. 
This definition also fixes
its boundary values and makes $G$ upper semicontinuous on
$\Delta$; see \cite[Section 1]{BC11}.
We define $G_D$ on $\Delta_D$ in the same way. 
Writing $\vol_d$ for $d$-dimensional Lebesgue measure
and $\E_\Delta$, $\E_{\Delta_D}$ for normalized Lebesgue averages, the above theorem yields  
\begin{equation}\label{eq-concave-transform-moments}
\begin{aligned}
&E^{\NA}(\Phi)=\E_\Delta G, \quad 
\norm{\Phi}_2^2=\E_\Delta(G-\E_\Delta G)^2, 
\quad \text{and}\quad
E_D=\E_{\Delta_D} G_D.
\end{aligned}
\end{equation}

The quotient filtration also gives the boundary comparison
\begin{equation}\label{eq-transform-restriction}
G_D(y)\leq G(0,y)\qquad y\in \Delta_D.
\end{equation}

It remains to estimate the difference of 
the two averages \eqref{eq-minimum-restriction-mean} on the cone
\eqref{eq-divisor-Okounkov-cone}. We first prove a one-variable inequality.
Let $R$ be a $[0,1]$-valued random variable whose distribution has
density $nr^{n-1}$ with respect to Lebesgue measure $dr$. Thus, for every
integrable function $f\colon [0,1]\to\R$,
\begin{equation*}
\E[f(R)]=n\int_0^1 f(r)r^{n-1}\,dr.
\end{equation*}
For every bounded convex function $h\colon [0,1]\to\R$, we claim that
\begin{equation}\label{eq-beta-convex-estimate}
h(1)-\E[h(R)]\geq
\frac{(n+1)(n+2)}{n}\operatorname{Cov}(h(R),R),
\end{equation}
where $\operatorname{Cov}(X, Y)\coloneqq\E[XY]-\E[X]\E[Y]$ and
$\operatorname{Var}(X)\coloneqq\operatorname{Cov}(X,X)$ for random variables.
To see this, let $F(h)$ be the left side of
\eqref{eq-beta-convex-estimate} minus its right side. Since
\begin{equation*}
\E[R]=\frac{n}{n+1},\qquad
\operatorname{Var}(R)=\frac{n}{(n+1)^2(n+2)},
\end{equation*}
$F$ vanishes on affine functions. For $0\leq t\leq1$, direct integration
with $(u)_+\coloneqq\max\{u,0\}$ gives
\begin{align*}
F((r-t)_+)
&=F((t-r)_+)\\
&=n\int_0^t(t-r)
 \left((n+1)-\frac{(n+1)(n+2)}{n}r\right)r^{n-1}\,dr
 =t^{n+1}(1-t)\geq0.
\end{align*}
The first equality uses $(r-t)_+-(t-r)_+=r-t$ and the vanishing of
$F$ on affine functions. A piecewise-linear convex function is affine plus a
nonnegative sum of these functions, so \eqref{eq-beta-convex-estimate}
holds for it. A bounded convex function has finite one-sided endpoint 
limits, and its continuous extension by these limits is uniformly
approximable by piecewise-linear convex functions. Its actual endpoint values
can only be larger: changing the value at $0$ does not affect the
expectations, and increasing the value at $1$ increases only the left
side. This proves \eqref{eq-beta-convex-estimate} as stated.

Now let $Y$ be a $\Delta_D$-valued random vector, independent of $R$,
whose distribution is normalized Lebesgue measure on $\Delta_D$.
The random variable 
\begin{equation*}
Z:= \bigl((1-R)/m,RY\bigr)
\end{equation*}
is uniformly distributed on $\Delta$. Indeed, the map
\begin{equation*}
T\colon [0,1]\times\Delta_D\longrightarrow\Delta,
\qquad T(r,y)=\bigl((1-r)/m,ry\bigr),
\end{equation*}
has Jacobian
\begin{equation*}
DT(r,y)=
\begin{pmatrix}-1/m&0\\y&rI_{n-1}\end{pmatrix},
\qquad \abs{\det DT(r,y)}=\frac{r^{n-1}}m.
\end{equation*}
Consequently
\begin{equation*}
\vol_n(\Delta)
=\int_0^1\int_{\Delta_D}\frac{r^{n-1}}m\,dy\,dr
=\frac{\vol_{n-1}(\Delta_D)}{mn}.
\end{equation*}
Independence and change of variables now give, for every integrable
$f\colon \Delta\to\R$,
\begin{align*}
\E[f(Z)]
&=\frac{n}{\vol_{n-1}(\Delta_D)}
  \int_0^1\int_{\Delta_D}
    f\bigl((1-r)/m,ry\bigr)r^{n-1}\,dy\,dr\\
&=\frac{mn}{\vol_{n-1}(\Delta_D)}\int_\Delta f(x)\,dx
=\E_\Delta f.
\end{align*}

For each fixed $y\in\Delta_D$, set
\begin{equation*}
g_y(r)\coloneqq G\bigl((1-r)/m,ry\bigr),
\qquad h_y(r)\coloneqq-g_y(r),\qquad 0\leq r\leq1.
\end{equation*}
The function $h_y$ is convex, since $G$ is concave and $T(\cdot,y)$
is affine. Applying \eqref{eq-beta-convex-estimate} to $h_y$ gives
\begin{equation}\label{eq-cone-fixed-y}
G(0,y)-\E_R[g_y(R)]
\leq\frac{(n+1)(n+2)}{n}
       \operatorname{Cov}_R(g_y(R),R).
\end{equation}
The subscript $R$ means that $y$ is held fixed while the expectation
or covariance is taken only with respect to $R$.
Averaging the left side over $Y$ gives
\begin{equation*}
\E_Y\bigl[G(0,Y)-\E_R[g_Y(R)]\bigr]
=\E_{\Delta_D}G(0,\cdot)-\E_\Delta G.
\end{equation*}
For the right side, independence implies that the mean of $R$ remains
$n/(n+1)$ for every fixed $y$. Fubini's theorem therefore gives
\begin{align*}
\E_Y\bigl[\operatorname{Cov}_R(g_Y(R),R)\bigr]
&=\E_Y\left[\E_R\left[g_Y(R)
             \left(R-\frac{n}{n+1}\right)\right]\right]\\
&=\E\left[G(Z)\left(R-\frac{n}{n+1}\right)\right]
=\operatorname{Cov}(G(Z),R). 
\end{align*}
Combining the
average of \eqref{eq-cone-fixed-y} with
\eqref{eq-transform-restriction}, and then applying Cauchy--Schwarz,
we obtain
\begin{align*}
\E_{\Delta_D}G_D-\E_\Delta G
&\underalign{\eqref{eq-transform-restriction}}{\leq}
  \E_{\Delta_D}G(0,\cdot)-\E_\Delta G
  \underalign{}{\leq}\frac{(n+1)(n+2)}{n}\operatorname{Cov}(G(Z),R)\\
&\underalign{}{\leq}\frac{(n+1)(n+2)}{n}
  \sqrt{\operatorname{Var}(G(Z))\operatorname{Var}(R)}
  \underalign{\eqref{eq-concave-transform-moments}}{=}
  \sqrt{\frac{n+2}{n}}\,\norm{\Phi}_2.
\end{align*}
This estimate with 
\eqref{eq-minimum-restriction-mean} proves the claim.
\end{proof}

\begin{rem}
Claim \ref{claim-minimal-l2} is optimal 
even for a product configuration of $\mathbb{P}^n$.
Choose the torus action with weights $0,1,\ldots,1$ 
on $H^0(\mP^n,\cO(1))$. 
A direct computation yields 
\begin{equation*}
\norm{\Phi}_{\min}=\frac{n}{n+1},\qquad
\norm{\Phi}_2^2=\frac{n}{(n+1)^2(n+2)} 
\end{equation*}
for the associated product configuration, 
which gives equality. 
\end{rem}

\begin{rem}
\label{rem-Bochner-vanishing}
Under the assumptions of Theorem \ref{thm-Donaldson-MY}, equality in
\eqref{eq-weak-MY} implies the existence of a sequence of K\"ahler
metrics $\omega_j\in c_1(X)$ which is asymptotically Bochner-flat
in the $L^2$ sense. 
Bryant \cite{Bry01} classified 
Bochner-flat K\"ahler manifolds and 
also observed that every weighted projective space admits a
Bochner-flat orbifold metric
\cite{Bry01}. 
Thus the weighted projective spaces $\mP(1,\ldots,1,p)$ appearing,
up to finite quasi-\'etale cover, in the equality statement of
Theorem \ref{thm-main-klt} carry such metrics.
This is consistent with the asymptotic vanishing of Bochner curvature
obtained above and provides a differential-geometric interpretation
of the equality case.
\end{rem}

\renewcommand{\thesection}{\Alph{section}}
\renewcommand{\theHsection}{appendix.\Alph{section}}
\setcounter{section}{0}

\section{
\texorpdfstring{
Alternative proof of Theorem \ref{thm-slope-gamma-klt}
by the approach of \cite{DGP24}
}{
Alternative proof by the approach of DGP24
}}
\label{sec-proof-DGP}

In this appendix, we briefly explain that Theorem \ref{thm-slope-gamma-klt} can be proved by the same approach as in \cite{DGP24}. More precisely, we prove the following statement.
\begin{thm}[=Theorem \ref{thm-slope-gamma-klt}]
\label{thm-slope-gamma-klt-2}
In the setting of Theorem \ref{thm-slope-gamma-klt}, let \(\mathscr{V}\) be the canonical extension sheaf whose extension class is \(\sqrt{\frac{2 \pi \delta'(X)}{n+1}}c_1(X)\). For any  torsion-free sheaf \(\mathcal{F} \subset \mathscr{V}\) of rank \(r\), 
$$
 c_1(\mathcal{F}) \cdot c_1(X)^{n-1} 
\le
\left( 1- \delta'(X) \left(1-\frac{r}{n+1}\right) \right)c_1(X)^{n}.
$$
\end{thm}

\subsubsection{A brief review of differential geometry}
\label{subsubsec-review-DG}

We use the Einstein summation convention throughout. For a K\"ahler form \(\omega\) and a \((1,1)\)-form \(\alpha\), we define
the \emph{trace operator} $ \tr_{\omega}$ and the \emph{\(\sharp\)-operator} by 
$$
 \tr_{\omega}(\alpha) 
  \coloneqq \frac{n \alpha \wedge \omega^{n-1}}{\omega^{n}}
 \quad
\text{and}
\quad
\sharp^{\omega} \alpha  \colon TX \to TX \quad \text{ such that }
\alpha(u, \bar{v})=\omega(\sharp^{\omega} \alpha(u), \bar{v}).
$$
If \(\omega=\sqrt{-1}g_{i \bar{j}} dz^i \wedge d \bar{z}^{\bar{j}}\) and \(\alpha = \sqrt{-1}\alpha_{i \bar{j}} dz^i \wedge d \bar{z}^{\bar{j}}\), then we can write
$$
\tr_{\omega}(\alpha)= g^{i\bar{j}}\alpha_{i\bar{j}}
\quad
\text{and}
\quad
\sharp^{\omega} \alpha
 \coloneqq  g^{i \bar{k}}\alpha_{j \bar{k}} \cdot 
dz^{j} \otimes \frac{\partial}{\partial z^i}
$$
Here \(g^{i \bar{k}}\) denotes the inverse matrix of \(g_{i \bar{k}}\).

Let \(\omega\) be a K\"ahler metric. Then \(T_X\) admits the Hermitian metric \(h\) induced by \(\omega\). The Chern curvature \(F_h = \bar{\partial}(h^{-1}\partial h)\) is defined as an \(\operatorname{End}(T_X)\)-valued \((1,1)\)-form. Locally, \(F_h\) can be written as
$$
F_h =R^{j}_{i k \bar{l}}
\cdot 
\frac{\partial}{\partial z^{j}}\otimes dz^i \otimes dz^k \wedge d\bar{z}^{\bar{l}}.
$$
Assume that \(\omega=\sqrt{-1}g_{i \bar{j}} dz^i \wedge d \bar{z}^{\bar{j}}\). 
We define 
$$
R_{i \bar{j}k \bar{l}} \coloneqq  g_{p\bar{j}} R^{p}_{i k \bar{l}},
\quad
R_{i\bar{j}} \coloneqq g^{k \bar{l}}R_{i \bar{j}k \bar{l}} 
\quad \text{and} \quad
\Ric(\omega) \coloneqq  \frac{\sqrt{-1}}{2 \pi} R_{i\bar{j}}dz^i \wedge d\bar{z}^{\bar{j}}.
$$
 With this convention, \(\Ric(\omega)\) is a real \((1,1)\)-form and \([\Ric(\omega)]= c_1(X)\) is an integral cohomology class in \(H^2(X, \Z)\). By \cite[(3.11)]{Gue16},  we obtain
\begin{equation}
\label{eq-sharp-formula}
n\cdot \frac{\sqrt{-1}}{2\pi}F_h\wedge \omega^{n-1}
= \sharp^{\omega}\Ric(\omega) \cdot \omega^{n}.
\end{equation}

\subsubsection{Existence of twisted K\"ahler--Einstein metrics}

\begin{thm}
\label{thm-delta-greatestricci-klt}
Let $X$ be a projective klt Fano variety, let $\theta\in c_1(X)$
be a smooth K\"ahler form, and let $\gamma$ be a rational number with
$0<\gamma<\delta'(X)$.
Then there exists a closed positive current $\omega\in c_1(X)$
with bounded potentials, smooth and K\"ahler on $X_{\mathrm{reg}}$, such that
\[
\Ric(\omega)=\gamma\omega+(1-\gamma)\theta
\quad\text{on }X_{\mathrm{reg}}.
\]
\end{thm}
This follows from \cite[Theorems 1.7 and 5.18]{HL23}, together with
the regularity results \cite[Corollary 3.6 and Proposition 3.7]{HL23}.

\begin{proof}
Set $\alpha\coloneqq(1-\gamma)\theta$ and $L\coloneqq-\gamma K_X$.
Since $\alpha$ is smooth, $v(\alpha)=0$ for every divisorial valuation $v$.
Hence
\[
\inf_v\frac{A_{0+\alpha}(v)}{S_L(v)}
=\frac{\delta(X)}{\gamma}>1.
\]
Taking trivial group actions and $g=1$ in
\cite[Theorems 1.7 and 5.18]{HL23}, we obtain a twisted
K\"ahler--Einstein current $\eta\in c_1(L)$ satisfying
$\Ric(\eta)=\eta+\alpha$ on $X_{\mathrm{reg}}$.
Its potentials are bounded, and it is smooth on $X_{\mathrm{reg}}$ by
\cite[Corollary 3.6 and Proposition 3.7]{HL23}.
Setting $\omega\coloneqq\gamma^{-1}\eta$ gives the assertion.
\end{proof}

\subsubsection{Setup}
Theorem \ref{thm-delta-greatestricci-klt} supplies metrics in the
following setup when $\theta$ is K\"ahler and $\gamma$ is rational.
We use the setup in the slightly more general form of
\cite[Assumption A]{DGP24}.

\begin{setup}
\label{setup-singular-Fano}
Let $X$ be an $n$-dimensional klt Fano variety. Let \(\omega_{X} \in c_1(X)\) be a K\"ahler metric on \(X\), and let $\theta \in c_1(X)$ be a smooth semipositive form. Fix $\gamma$ satisfying \(0<\gamma<\delta'(X)\).
Assume that $X$ admits a twisted Kähler--Einstein metric $\omega$ relative to $(\theta,1-\gamma)$. In other words, assume that there exists a Kähler metric $\omega$ satisfying the following conditions:
 \begin{enumerate}
 \item $\omega$ is a closed positive $(1,1)$-current with bounded potentials in $c_1(X)$.
 \item $\omega$ is smooth on $X_{\mathrm{reg}}$ and
\(
\Ric(\omega)=\gamma \omega +(1-\gamma) \theta
\text{ on $X_{\reg}$}.
\)
\end{enumerate}
\end{setup}

By using the argument in \cite[Lemma 10]{DGP24} and letting rational $\gamma$ tend to $\delta'(X)$, it is enough to prove the following theorem.

\begin{thm}
\label{thm-slope-2}
Assume Setup \ref{setup-singular-Fano}. 
Let $\pi \colon \widetilde{X} \to X$ be a resolution of singularities and  let $\widetilde{\mathscr V}$ be the canonical extension sheaf on $\widetilde X$
\begin{equation}
\label{eq-canonical-extension-2pidelta}
0 \longrightarrow \mathcal{O}_{\widetilde{X}} \longrightarrow  \widetilde{\mathscr V} \longrightarrow  T_{\widetilde{X}} \longrightarrow 0
\end{equation}
whose extension class is
$\pi^*\left(\sqrt{\frac{2 \pi \delta'(X)}{n+1}}c_1(X)\right)\in H^1(\widetilde X,\Omega_{\widetilde X}^1)$.
Then, for any  torsion-free sheaf $\widetilde{\mathcal{F}} \subset \widetilde{\mathscr V}$ of rank $r$, 
\begin{equation}
\label{eq-estimate-widetildeX}
 c_1(\widetilde{\mathcal{F}}) \cdot \pi^{*}c_1(X)^{n-1} 
\le
\left(1- \gamma + \frac{\delta'(X) r}{n+1}\right)\cdot \pi^{*}c_1(X)^{n}.
\end{equation}
\end{thm}

\subsubsection{Notation and choice of approximations}
\label{subsubsec-notation-choice}
From now on, assume Setup \ref{setup-singular-Fano} and fix a subsheaf \(\widetilde{\mathcal{F}} \subset \widetilde{\mathscr V}\). We may assume that \(\widetilde{\mathcal{F}}\) is saturated in \(\widetilde{\mathscr V}\). We will prove Theorem \ref{thm-slope-2}, that is, we will establish the estimate in \eqref{eq-estimate-widetildeX}.

We define the necessary objects as in \cite[Notation 5]{DGP24}.
\begin{itemize}
\item Let $\pi \colon \widetilde{X} \to X$ be a resolution of singularities, and let \(
E=\sum_{k}E_k
\)
be the exceptional divisor of \(\pi\). By the klt condition, each discrepancy \(a_k\), defined by
\(
K_{\widetilde{X}}=\pi^*K_X+\sum_{k} a_kE_k, 
\)
satisfies $a_k>-1$.
\item Choose \(\varepsilon_k\in\mathbb{Q}_+\) such that \(\pi^*c_1(X)-\sum_{k}\varepsilon_k c_1(E_k)\) is ample, and take a K\"ahler metric
\(
\omega_{\widetilde{X}} \in \pi^*c_1(X)-\sum_{k}\varepsilon_k c_1(E_k).
\)

\item For each \(k\), take a section
\(
s_k\in H^0\bigl(\widetilde{X},\mathcal{O}_{\widetilde{X}}(E_k)\bigr)
\)
such that \(E_k=(s_k=0)\). 
Let \(h_k\) be a smooth Hermitian metric on \(\mathcal{O}_{\widetilde{X}}(E_k)\).  

\item Choose a volume form \(dV\) on \(\widetilde{X}\) such that
\(
\Ric(dV)=\pi^*\omega_X-\sum_{k} a_k \cdot \frac{\sqrt{-1}}{2\pi}F_{E_k, h_k}.
\)

\item Take a smooth function \(u\) on \(\widetilde X\) such that
\(
\pi^*\theta = \pi^*\omega_{X} + \frac{\sqrt{-1}}{2\pi} \partial \bar{\partial} u.
\)
\end{itemize}
Now take \(\varphi\) such that
\(
\pi^*\omega=\pi^*\omega_X+ \frac{\sqrt{-1}}{2\pi} \partial \bar{\partial} \varphi
\)
as in \cite[Subsubsection 2.1.2]{DGP24}.
By Setup \ref{setup-singular-Fano} (2), \(\varphi\) is  the solution of
$$
\left(\pi^*\omega_X+ \frac{\sqrt{-1}}{2\pi} \partial \bar{\partial}\varphi\right)^n
= e^{-\gamma\varphi -(1-\gamma)u} \cdot \prod_{k} \lvert s_k\rvert^{2a_k}_{h_k} \cdot dV.
$$
Thus \(\varphi\) is a continuous quasi-psh function on $\widetilde{X}$ smooth outside $E$ (cf. \cite[Section 3]{BBEGZ19}). 
By applying Demailly's regularization theorem to $\varphi$, we obtain a family of quasi-psh functions $\psi_\varepsilon\in\mathscr{C}^\infty(\widetilde{X})$ with $\psi_\varepsilon\to\varphi$.
As in \cite[Section 2.1.2]{DGP24}, we choose the regularizations
uniformly bounded, converging to $\varphi$ in $L^1$ and in
$\mathscr C^\infty_{\mathrm{loc}}(\widetilde X\setminus E)$, and satisfying
\[
\pi^*\omega_X+\frac{\sqrt{-1}}{2\pi}\partial\bar\partial\psi_\varepsilon
\ge -\kappa(\varepsilon)\omega_{\widetilde X},
\qquad \kappa(\varepsilon)\longrightarrow0.
\]
Using \(\psi_\varepsilon\), for any \(\varepsilon,t> 0\), we can take a solution
\(
\varphi_{t,\varepsilon}\in L^\infty(\widetilde{X})\cap \mathrm{PSH}\bigl(\widetilde{X},\pi^*\omega_X+t\omega_{\widetilde{X}}\bigr)
\)
of
\begin{equation}
\label{eq-ricci-omega-t-epsilon}
\left(\pi^*\omega_X+t\omega_{\widetilde{X}}+\frac{\sqrt{-1}}{2\pi} \partial \bar{\partial} \varphi_{t,\varepsilon}\right)^n
= f_\varepsilon\,e^{-\gamma\psi_\varepsilon -(1-\gamma)u}\,e^{-c_t}\,dV
\quad \text{and} \quad 
\sup_{\widetilde{X}}\varphi_{t,\varepsilon}=0.
\end{equation}
Here \(f_\varepsilon \coloneqq e^{ b_\varepsilon}\prod_{k} (\lvert s_k\rvert_{h_k}^2+\varepsilon^2)^{a_k}\), and \( b_\varepsilon\) and \(c_t\) are normalization constants as in \cite[Page 98]{DGP24}.
As in \cite[(5)]{DGP24}, set
\begin{equation}
\label{eq-defn-ric-tvarepsilon}
\omega_{t,\varepsilon} \coloneqq 
\pi^*\omega_X+t\omega_{\widetilde{X}}+\frac{\sqrt{-1}}{2\pi} \partial \bar{\partial}\varphi_{t,\varepsilon}  
\in \pi^{*}c_1(X) + t \pi^*c_1(X)-t \sum_{k}\varepsilon_k c_1(E_k). 
\end{equation}
Equation \eqref{eq-ricci-omega-t-epsilon} gives
\begin{equation}
\label{eq-DGP24-7}
\begin{aligned}
\Ric(\omega _{t,\varepsilon})
&\underalign{\eqref{eq-ricci-omega-t-epsilon}}{=}
\gamma \omega_{t,\varepsilon}
+(1-\gamma)\pi^{*}\theta-\gamma t\omega_{\widetilde{X}} +
\gamma\,\frac{\sqrt{-1}}{2\pi} \partial \bar{\partial}(\psi_\varepsilon-\varphi_{t,\varepsilon})-\frac{\sqrt{-1}}{2\pi} F_\varepsilon.
\end{aligned}
\end{equation}
Here $F_\varepsilon$ is defined as follows:
$$
\frac{\sqrt{-1}}{2\pi} F_\varepsilon
  \coloneqq \sum_{k} a_k\,\left(\frac{\sqrt{-1}}{2\pi}F_{E_k, h_k} +\frac{\sqrt{-1}}{2\pi} \partial \bar{\partial}\log(\lvert s_k\rvert_{h_k}^2+\varepsilon^2) \right)\in \sum_{k} a_k c_1(E_k) 
$$
which converges to the integration current along $ \sum_{k} a_k [E_k]$ as $\varepsilon\to 0$.

\begin{obs}
When $X$ is smooth, this construction simplifies as follows:
\begin{itemize}
\item $\pi=\id_{X}$ and $\omega_{\widetilde{X}}=\omega_X$.  Since $E=0$, the quantities $a_k, s_k, \varepsilon_k$ can be ignored;
\item $\omega= \omega_{X} + \frac{\sqrt{-1}}{2\pi} \partial \bar{\partial} \varphi$. Since $\varphi$ is smooth, all $\psi_{\varepsilon}, \varphi_{t, \varepsilon}, b_\varepsilon, c_t$ can be ignored. In particular, $t$ and $\varepsilon$ can also be ignored.
\end{itemize}
\end{obs}

\subsubsection{Construction of the canonical extension sheaf}
\label{subsubsec-const-canonicalext}
We follow the argument in \cite[Subsection 3.3, Step 1]{DGP24}. Set 
\[\alpha  \coloneqq \sqrt{\frac{2 \pi \delta'(X)}{n+1}}\pi^{*}c_1(X) \in H^1(\widetilde{X}, \Omega_{\widetilde{X}}^{1}).\] 
We can explicitly construct the canonical extension sheaf \(\widetilde{\mathscr V}\) as in \eqref{eq-canonical-extension-2pidelta}. 
Take a sufficiently small open cover \(\{U_i\}\) such that
\(
\widetilde{\mathscr V}\vert_{U_i}=\mathcal O_{\widetilde{X}}\vert_{U_i}\oplus T_{\widetilde{X}}\vert_{U_i}.
\)
Regarding \(\alpha\) as an element of \(H^1\bigl(\widetilde{X},\Homop(T_{\widetilde{X}},\mathcal O_{\widetilde{X}})\bigr)\), we write it as a \v{C}ech \(1\)-cocycle \(\alpha=(a_{ij})\), where each \(a_{ij}\) is a local section of \(\Homop(T_{\widetilde{X}},\mathcal O_{\widetilde{X}})\) on \(U_{ij} \coloneqq U_i\cap U_j\).
With this representation, the transition functions of \(\widetilde{\mathscr V}\) on \(U_{ij}\) are given by
\[
\begin{pmatrix}
\Id_{\mathcal O_{\widetilde X}\vert_{U_{ij}}} & a_{ij}\\
0 & \Id_{T_{\widetilde X}\vert_{U_{ij}}}
\end{pmatrix}.
\]
For any $t\in\mathbb{R}_{>0}$, define an element of $H^1(\widetilde{X}, \Omega_{\widetilde{X}}^{1})$ by
\begin{align*}
\alpha_{t}
& \coloneqq \sqrt{\frac{2 \pi \delta'(X)}{n+1}}\left(  \pi^{*}c_1(X) - \frac{t}{1+t}\sum_k \varepsilon_k c_1(E_k)\right).
\end{align*}
Let $\widetilde{\mathscr V}_{t}$ be the extension of $T_{\widetilde X}$ by $\mathcal O_{\widetilde X}$ whose extension class is $\alpha_{t}$. We also set
\(
\widetilde{\mathscr{V}}_t(E) \coloneqq \widetilde{\mathscr{V}}_t\otimes\mathcal{O}_{\widetilde{X}}(E).
\)
If $t$ is sufficiently small, 
we obtain $\widetilde{\mathcal{F}}\subset \widetilde{\mathscr{V}}_t(E)$ by \cite[Subsection 3.3, Step 1]{DGP24}.

\subsubsection{Construction of a metric on $\widetilde{\mathscr{V}}_t$}
We follow the argument in \cite[Subsection 3.3, Step 2]{DGP24}. 
We identify the extension with class $\alpha_t$ with that of
class $(1+t)\alpha_t$ by rescaling the $\mathcal O_{\widetilde X}$-summand.
Fix a \(C^\infty\) splitting \(\widetilde{\mathscr{V}}_t \underset{C^{\infty}}{\cong}\mathcal{O}_{\widetilde{X}} \oplus T_{\widetilde{X}}\). Set
\[
\beta_{t} \coloneqq \sqrt{\frac{2 \pi \delta'(X)}{n+1}}\omega_{t,\varepsilon}
\underalign{\eqref{eq-defn-ric-tvarepsilon}}{\in} (1+t)\alpha_{t},
\]
and regard it as a $\Homop(T_{\widetilde{X}},\mathcal O_{\widetilde{X}})$-valued $(0,1)$-form.
Thus \(\widetilde{\mathscr{V}}_t\) admits the structure of a holomorphic vector bundle.

We put the trivial metric on \(\mathcal{O}_{\widetilde{X}}\) and the Hermitian metric \(h_{t,\varepsilon}\) induced by the K\"ahler metric \(\omega_{t,\varepsilon}\) on \(T_{\widetilde{X}}\). Then the direct sum metric on \(\widetilde{\mathscr{V}}_t\) induces a metric \(h_{\widetilde{\mathscr{V}}_t}\). We denote by 
\(\beta_{t}^*\) the adjoint of \(\beta_{t}\).  $\beta_{t}^*$ is regarded as a $T_{\widetilde{X}}$-valued $(1,0)$-form.
If  \(\beta_{t}\) is locally written as
\begin{equation}
\label{eq-beta-defn}
\beta_{t}
=
\sqrt{\frac{2 \pi \delta'(X)}{n+1}}\cdot 
\sqrt{-1}\sum_{i, j} \omega_{i\bar j}\Bigl(\frac{\partial}{\partial z^i}\Bigr)^{\vee}\otimes d\bar z^{\bar{j}},
\end{equation}
where \(\omega_{i\bar j}\) are the coefficients of \(\omega_{t,\varepsilon}\) with respect to the coordinates \((z^1, \ldots, z^n)\), then the adjoint \(\beta_{t}^*\) is computed as
\begin{equation}
\label{eq-beta-star}
\beta_{t}^*=-
\sqrt{\frac{2 \pi \delta'(X)}{n+1}}
\cdot \sqrt{-1}\sum_{i} \frac{\partial}{\partial z^i}\otimes dz^i
\end{equation}
This follows from the pairing 
\(
\langle \beta_{t}\cdot v,w\rangle_{\mathcal{O}_{\widetilde{X}}}
=\langle v,\beta_{t}^*\cdot w\rangle_{T_{\widetilde{X}}, h_{t, \varepsilon}}
\)
induced by the definition of the adjoint. 
Noting that \(\beta_t\) is a \((0,1)\)-form and applying \cite[Lem. 11.2]{Dem12}, the Chern curvature of \(h_{\widetilde{\mathscr{V}}_t}\) is given by 
\[
F_{\widetilde{\mathscr{V}}_t,h_{\widetilde{\mathscr{V}}_t}}
=
\begin{pmatrix}
-\beta_{t}\wedge\beta_{t}^* 
& -D'_{\Homop(T_{\widetilde{X}},\mathcal O_{\widetilde{X}}), h_{t, \varepsilon}}\beta_{t}\\
\bar\partial\beta_{t}^* 
& F_{T_{\widetilde{X}},h_{t,\varepsilon}}-\beta_{t}^*\wedge\beta_{t}
\end{pmatrix}
\,
\underalign{\eqref{eq-beta-defn} \text{ \& } \eqref{eq-beta-star}}{=}
\,
\begin{pmatrix}
-\beta_{t}\wedge\beta_{t}^* 
& 0\\
0
& F_{T_{\widetilde{X}},h_{t,\varepsilon}}-\beta_{t}^*\wedge\beta_{t}
\end{pmatrix}.
\]
Thus, using \eqref{eq-beta-defn} and \eqref{eq-beta-star}, it follows that
\begin{equation}
\label{eq-curv-calc-matrix}
\frac{\sqrt{-1}}{2\pi}F_{\widetilde{\mathscr{V}}_t,h_{\widetilde{\mathscr{V}}_t}}
\wedge\omega_{t, \varepsilon}^{n-1}
\,
\underalign{\eqref{eq-beta-defn} \text{ \& } \eqref{eq-beta-star}}{=}
\,
\begin{pmatrix}
\frac{\delta'(X)}{n+1} \,\omega_{t, \varepsilon} ^n & 0\\
0 
& \frac{\sqrt{-1}}{2\pi}F_{T_{\widetilde{X}},h_{t, \varepsilon}}\wedge\omega_{t, \varepsilon}^{n-1}
-\frac{\delta'(X)}{n(n+1)}
\,\omega_{t, \varepsilon}^n\otimes \Id_{T_{\widetilde{X}}}
\end{pmatrix}.
\end{equation}
Notice that the difference in the coefficients from \cite[(27) below]{DGP24} is due to our convention.
We compute the lower-right component of \eqref{eq-curv-calc-matrix} as follows:
\begin{equation}
\begin{aligned}
\label{eq-curv-calc}
&
\frac{\sqrt{-1}}{2\pi}F_{T_{\widetilde{X}},h_{t, \varepsilon}} \wedge\omega_{t, \varepsilon}^{n-1}
-
\frac{\delta'(X)}{n(n+1)} \,\omega_{t, \varepsilon}^n\otimes \Id_{T_{\widetilde X}}\underalign{\text{\eqref{eq-sharp-formula}}}{=}
\frac{1}{n}\,\sharp^{\omega_{t,\varepsilon}}\Ric(\omega_{t, \varepsilon})\cdot \omega_{t, \varepsilon}^{n} 
-\frac{\delta'(X)}{n(n+1)} \,\omega_{t, \varepsilon}^n\otimes \Id_{T_{\widetilde X}} \\
&\underalign{\eqref{eq-DGP24-7}}{=}
\frac{1}{n}\,\sharp^{\omega_{t,\varepsilon}}(\gamma \omega_{t,\varepsilon}
+\underbrace{(1-\gamma)\pi^{*}\theta-\gamma t\omega_{\widetilde{X}} +
\gamma\,\frac{\sqrt{-1}}{2\pi} \partial \bar{\partial}(\psi_\varepsilon-\varphi_{t,\varepsilon})-\frac{\sqrt{-1}}{2\pi} F_\varepsilon  
}_{\eqqcolon A_{t, \varepsilon}})\cdot \omega_{t, \varepsilon}^{n} 
-\frac{\delta'(X)}{n(n+1)} \,\omega_{t, \varepsilon}^n\otimes \Id_{T_{\widetilde X}} \\
&\underalign{\text{(by $\sharp^{\omega_{t,\varepsilon}}\omega_{t,\varepsilon}=\Id_{T_{\widetilde X}}$)} }{=}\qquad
\frac{\delta'(X)}{n+1} 
\,\omega_{t, \varepsilon}^n\otimes \Id_{T_{\widetilde X}}
+
\frac{(\gamma - \delta'(X))}{n} \,\omega_{t, \varepsilon}^n\otimes \Id_{T_{\widetilde X}}
+\frac{1}{n}\,\sharp^{\omega_{t,\varepsilon}} A_{t, \varepsilon}\cdot \omega_{t, \varepsilon}^{n}.
\end{aligned}
\end{equation}

\subsubsection{Slope estimate for $\widetilde{\mathcal{F}}$}
As in Subsection \ref{subsubsec-const-canonicalext}, take $t$ sufficiently
small so that $\widetilde{\mathcal F}\subset\widetilde{\mathscr V}_t(E)$.
For the upper slope bound, we may replace $\widetilde{\mathcal F}$ by its
saturation in $\widetilde{\mathscr V}_t(E)$, since this can only increase
its slope; see \cite[Subsection 3.3, Step 3]{DGP24}.
Let $W$ be the open subset on which both this saturated subsheaf and
its quotient are locally free, and set
$\widetilde F\coloneqq\widetilde{\mathcal F}|_W$.
Then $\widetilde F\subset\widetilde{\mathscr V}_t(E)|_W$ is a subbundle.
By Subsection \ref{subsubsec-notation-choice}, \(\mathcal O_{\widetilde X}(E)\) is equipped with the smooth Hermitian metric \(h_{E}\coloneqq \prod_{k} h_k\). We define Hermitian metrics \(h_{\widetilde{\mathscr{V}}_t(E)}\) on \(\widetilde{\mathscr{V}}_t(E)\) and \(h_{\widetilde{F}}\) on \(\widetilde{F}\) by
$$
h_{\widetilde{\mathscr{V}}_t(E)}
\coloneqq h_{\widetilde{\mathscr{V}}_t} \cdot h_{E}
\quad \text{and} \quad 
h_{\widetilde{F}}
\coloneqq h_{\widetilde{\mathscr{V}}_t(E)}|_{\widetilde{F}},
$$
respectively.
Let \(b_{t, \varepsilon}\) be the second fundamental form of \(\widetilde{\mathcal{F}} \subset \widetilde{\mathscr{V}}_t(E)\) with respect to \(h_{\widetilde{\mathscr{V}}_t(E)}\) on \(W\). Note that \(b_{t, \varepsilon}\) is a \(\Homop(\widetilde{F}^{\perp},\widetilde{F})\)-valued \((0,1)\)-form on \(W\), and that our notation differs from that in \cite[Lem. 11.2]{Dem12}. 
Combining the above arguments,

\begin{align*}
& c_1(\widetilde{\mathcal{F}}) \cdot
\{\omega_{t, \varepsilon}\}^{n-1}
\underalign{\raisebox{-0.8ex}{\tiny\text{\cite[Subsec.3.3, Step 3]{DGP24}}}}{\le}
\qquad
\int_{W} \frac{\sqrt{-1}}{2\pi} \tr\,F_{\widetilde{F}, h_{\widetilde{F}}} \wedge \omega_{t, \varepsilon}^{n-1} \\
&\quad
\underset{\mathclap{
\begin{subarray}{c}
\text{\cite[(10)]{DGP24}}\\
\text{\cite[Lem. 11.2]{Dem12}}
\end{subarray}
}}{=}
\qquad
\int_{W} 
\tr \left(\pr_{\widetilde{F}}\left(
 \frac{\sqrt{-1}}{2\pi}
 \left.F_{\widetilde{\mathscr{V}}_t(E),h_{\widetilde{\mathscr{V}}_t(E)}} \right\vert_{\widetilde{F}}\right)
\right)
\wedge \omega_{t, \varepsilon}^{n-1} 
+\int_{W} \frac{\sqrt{-1}}{2\pi} \tr\left(b_{t, \varepsilon}\wedge b_{t, \varepsilon}^{*}\wedge \omega_{t, \varepsilon}^{n-1}\right) 
\\
&
\underalign{\raisebox{-2.0ex}{\tiny\text{\eqref{eq-curv-calc-matrix} \& \eqref{eq-curv-calc}}}}{=}
\underbrace{\frac{\delta'(X)}{n+1} \int_{W}
 \tr\left(\pr_{\widetilde{F}}\left(\Id_{\widetilde{\mathscr{V}_t}}\vert_{\widetilde{F}}\right)\right)\omega_{t, \varepsilon}^{n}}_{=\frac{\delta'(X) \cdot r}{n+1} \{ \omega_{t, \varepsilon}\}^{n}}
 + 
\underbrace{
\frac{(\gamma - \delta'(X))}{n}
\int_{W}
 \tr\left(\pr_{\widetilde{F}}
 \left( (0_{\mathcal{O}_{\widetilde X}} \oplus \Id_{T_{\widetilde X}})\vert_{\widetilde{F}} \right)\right)\omega_{t, \varepsilon}^{n}
 }_{
 \le 0 \text{ by $\gamma < \delta'(X)$}
 }
 \\
&\underalign{}{+}
\underbrace{ \frac{1}{n}\int_{W}
 \tr\left(\pr_{\widetilde{F}}
 \left( (0_{\mathcal{O}_{\widetilde X}} \oplus \sharp^{\omega_{t, \varepsilon}}A_{t, \varepsilon})\vert_{\widetilde{F}} \right)\right)\omega_{t, \varepsilon}^{n}}_{\eqqcolon a_{t, \varepsilon}} 
  \\
&\underalign{}{+}
\underbrace{\int_{W} 
\tr \left(\pr_{\widetilde{F}}\left(
 \frac{\sqrt{-1}}{2\pi}
 \left.F_{E, h_{E}} \otimes \Id_{\widetilde{\mathscr{V}}_t(E)}\right\vert_{\widetilde{F}}\right)
\right)
\wedge \omega_{t, \varepsilon}^{n-1} 
}_{\eqqcolon c_{t, \varepsilon}}
+\underbrace{\int_{W} \frac{\sqrt{-1}}{2\pi}\tr 
\left(b_{t, \varepsilon}\wedge b_{t, \varepsilon}^{*}\wedge \omega_{t, \varepsilon}^{n-1}\right) }_{\le 0 \text{ (cf. \cite[The term (IV)]{DGP24}, \cite[Cor. 11.8]{Dem12})}}.
\end{align*}
Here, \(0_{\mathcal{O}_{\widetilde{X}}}\) denotes the zero map on $\mathcal{O}_{\widetilde{X}}$. Since $\omega_{t, \varepsilon} \in (1+t)\pi^*c_1(X) -t \sum_{k}\varepsilon_k c_1(E_k)$ by \eqref{eq-defn-ric-tvarepsilon}, 
\begin{equation}
\label{eq-degree-estimate}
c_1(\widetilde{\mathcal{F}}) \cdot 
\left((1+t)\pi^*c_1(X) -t \sum_{k}\varepsilon_k c_1(E_k)\right)^{n-1}
\le 
\frac{\delta'(X) r}{n+1} \left( (1+t)\pi^*c_1(X) -t \sum_{k}\varepsilon_k c_1(E_k) \right)^{n}
+ a_{t, \varepsilon}
+ c_{t, \varepsilon}.
\end{equation}
As for $c_{t, \varepsilon}$, since $E$ is an exceptional divisor,
\begin{equation}
\label{eq-bt-epsilon-0}
c_{t, \varepsilon}
=
r
\int_{W} 
\tr\left( \frac{\sqrt{-1}}{2\pi}F_{E, h_{E}} \right)
\wedge \omega_{t, \varepsilon}^{n-1} 
=
r c_1(E) \cdot \left((1+t)\pi^*c_1(X) -t \sum_{k}\varepsilon_k c_1(E_k)\right)^{n-1}
\underset{t \to 0}{\longrightarrow} 0.
\end{equation}

The term \(a_{t,\varepsilon}\) can be estimated by the following claim.
\begin{claim}
\label{claim-eval-a}
Ignoring $0_{\mathcal{O}_{\widetilde{X}}}$ and expanding $A_{t,\varepsilon}$ in \eqref{eq-curv-calc}, we decompose \(a_{t,\varepsilon}\) as follows:
\begin{align*}
a_{t, \varepsilon} 
&\underalign{}{ \coloneqq }\frac{1}{n}\int_{W}
 \tr\!\Bigl(\pr_{\widetilde{F}}(0_{\mathcal{O}_{\widetilde{X}}} \oplus \sharp^{\omega_{t, \varepsilon}}A_{t, \varepsilon})\vert_{\widetilde{F}} \Bigr)\omega_{t, \varepsilon}^{n} \\
 &\underalign{}{=} 
 \underbrace{\frac{(1-\gamma)}{n}\int_{W}
 \tr\!
 \Bigl(\pr_{\widetilde{F}}
 \left(\sharp^{\omega_{t, \varepsilon}}\pi^{*}\theta \vert_{\widetilde{F}} \right)\Bigr)\omega_{t, \varepsilon}^{n}}_{\text{Term (II)}}
 +
\underbrace{\frac{-\gamma t}{n}\int_{W}
 \tr\!
 \Bigl(\pr_{\widetilde{F}}\left(\sharp^{\omega_{t, \varepsilon}}\omega_{\widetilde{X}}\vert_{\widetilde{F}}\right) \Bigr)\omega_{t, \varepsilon}^{n}}_{\text{Term (0)}}  \\
&\underalign{}{+}
\underbrace{\frac{\gamma}{2\pi n}\int_{W}
 \tr\!
 \left(\pr_{\widetilde{F}}
 \left(\sharp^{\omega_{t, \varepsilon}} \sqrt{-1}\partial \bar{\partial}(\psi_\varepsilon-\varphi_{t,\varepsilon}) \vert_{\widetilde{F}}\right)\right)\omega_{t, \varepsilon}^{n}}_{\text{Term (I)}}
  +
\underbrace{\frac{-1}{2\pi n}\int_{W}
 \tr\!
 \left(\pr_{\widetilde{F}}
 \left(\sharp^{\omega_{t, \varepsilon}}
 \sqrt{-1}F_\varepsilon \vert_{\widetilde{F}}\right) \right)\omega_{t, \varepsilon}^{n}}_{\text{Term (III)}}.
\end{align*}
Then the following estimates hold:
\begin{itemize}
\item Term (0) satisfies $(0) \to 0$ as $t \to 0$.
\item Term (I) satisfies $\mathrm{(I)} \to 0$ as $\varepsilon, t \to 0$.
\item Term (II) satisfies $\limsup_{t,\varepsilon\to0}\mathrm{(II)}\le(1-\gamma)c_1(X)^n$. 
\item Term (III) satisfies $\mathrm{(III)} \to 0$ as $\varepsilon, t \to 0$.
\end{itemize}
In particular, $\limsup_{t, \varepsilon \to 0}a_{t, \varepsilon} \le (1-\gamma)c_1(X)^{n}$.
\end{claim}
By using Claim \ref{claim-eval-a} and \eqref{eq-bt-epsilon-0} and by letting $t, \varepsilon \to 0$ in \eqref{eq-degree-estimate}, we conclude that
$$
c_1(\widetilde{\mathcal{F}}) \cdot \pi^*c_1(X)^{n-1}
\le 
\left( 1- \gamma+ \frac{\delta'(X) r}{n+1} \right) c_1(X)^n.
$$
Thus we obtain the desired estimate of \eqref{eq-estimate-widetildeX} in Theorem \ref{thm-slope-2}.

\subsubsection{Proof of Claim \ref{claim-eval-a}}
Thus it remains only to prove Claim \ref{claim-eval-a}.
The estimates of Term (I) and Term (III) are the same as in \cite[The term (I) and (III)]{DGP24}. We therefore omit the proof. 

We focus on Term (II) and Term (0). Since $\pi^*\theta\ge 0$, Term (II) satisfies
\begin{equation}
\label{eq-projection-sharpomega1}
\tr\left(
\pr_{\widetilde{F}}(\sharp^{\omega_{t, \varepsilon}}\pi^*\theta\vert_{\widetilde{F}})\right)\omega_{t,\varepsilon}^n
\underalign{\text{(Property of projection)}}{\le} \tr(\sharp^{\omega_{t, \varepsilon}}\pi^*\theta)\,\omega_{t,\varepsilon}^n
=
\tr_{\omega_{t,\varepsilon}}(\pi^*\theta) \omega_{t,\varepsilon}^n
= n\,\pi^*\theta\wedge \omega_{t,\varepsilon}^{n-1}.
\end{equation}
Integrating over $\widetilde X$, we obtain
\begin{align*}
\mathrm{(II)} 
&\underalign{\eqref{eq-projection-sharpomega1}}{\le} 
(1-\gamma) \int_{W}
\pi^*\theta\wedge \omega_{t,\varepsilon}^{n-1}
\\
&\underalign{\text{(Setup \ref{setup-singular-Fano} \& \eqref{eq-defn-ric-tvarepsilon})}}{=} 
 (1-\gamma)\pi^{*}c_1(X) \cdot 
 \left(\pi^{*}c_1(X)  + t \pi^*c_1(X)- t\sum_{k}\varepsilon_k c_1(E_k) \right)^{n-1}.
\end{align*}
Thus $\mathrm{(II)} \le (1-\gamma)c_1(X)^{n}$ as $t \to 0$. 
The same argument applies to term (0):
\begin{align*}
\lvert\text{(0)}\rvert
&\underalign{\text{\eqref{eq-projection-sharpomega1}}}{\le} 
\gamma t
\left\lvert \int_{W} \omega_{\widetilde{X}} \wedge \omega_{t, \varepsilon}^{n-1}\right\rvert
\\
&\underalign{\raisebox{-1.2ex}{\tiny \text{(Sec.\ref{subsubsec-notation-choice} \& \eqref{eq-defn-ric-tvarepsilon})}}}{=} 
 \gamma t 
 \left\lvert 
\left(\pi^*c_1(X)-\sum_{k}\varepsilon_k c_1(E_k)\right) 
\cdot 
\left(\pi^{*}c_1(X)  + t \pi^*c_1(X)- t\sum_{k}\varepsilon_k c_1(E_k) \right)^{n-1}
 \right\rvert
\end{align*}
Thus $(0) \to 0$ as $t \to 0$. This completes the proof.

\bibliographystyle{alpha}
\bibliography{ref_MY_1_kltFano.bib}

\end{document}